\documentclass[11pt]{amsart}

\usepackage{amsmath} 
\usepackage{amssymb} 
\usepackage{amscd} 
\usepackage{tabularx} 
\usepackage[all,cmtip,line]{xy}
\usepackage{enumerate}
\usepackage{hyperref}
\usepackage{bm}

\usepackage{amsthm}
\newtheorem{thm}{\textbf{Theorem}}[section]
\newtheorem{defn}[thm]{\textbf{Definition}}
\newtheorem{prop}[thm]{\textbf{Proposition}}
\newtheorem{lem}[thm]{\textbf{Lemma}}
\newtheorem{corollary}[thm]{\text{Corollary}}
\newtheorem{rem}[thm]{\textbf{Remark}}

\newtheorem*{thmn}{\textbf{Theorem}}

\def\Q{\mathbb{Q}}
\def\Z{\mathbb{Z}}
\def\C{\mathbb{C}}
\def\A{\mathbb{A}}
\def\R{\mathbb{R}}

\def\GL{\operatorname{GL}}
\def\SL{\operatorname{SL}}
\def\PGL{\operatorname{PGL}}

\def\D{\mathcal{D}}
\def\Qbar{\overline{\mathbb{Q}}}

\newcommand{\Hom}{\operatorname{Hom}}

\usepackage{color}
\definecolor{Judith}{rgb}{1,0.2,0.2}

\title{A $p$-adic Labesse--Langlands transfer}
\author{Judith Ludwig}
\email{ludwig@mpim-bonn.mpg.de}

\begin{document}

\begin{abstract}
We prove a $p$-adic Labesse--Langlands transfer from the group of units in a definite quaternion algebra to its subgroup of norm one elements. 
More precisely, given an eigenvariety for the first group, we show that there exists an eigenvariety for the second group and a morphism between them that extends the classical Langlands transfer. In order to find a suitable target eigenvariety for the transfer we formalise a notion of Langlands compatibility of tame levels. Proving the existence of Langlands compatible tame levels is the key to pass from the classical transfer on the level of $L$-packets to a map between classical points of eigenvarieties, which is then amenable for interpolation to give the $p$-adic transfer.   
\end{abstract}
\maketitle
 
\section{Introduction}
In recent years there has been a lot of progress towards solving the Langlands functoriality conjectures (cf.\ e.g.\ \cite{Arthur} and \cite{Mok}). 
Furthermore a $p$-adic Langlands programme started to emerge. Although the general definition of a $p$-adic automorphic representation is still missing we have good working definitions of $p$-adic automorphic forms in many situation. Given such a definition one can ask which aspects of the classical Langlands programme make sense for these $p$-adic automorphic forms and it is particularly interesting to ask about Langlands functoriality: If $G$ and $H$ are two connected reductive groups defined over a number field together with a classical Langlands transfer from $G$ to $H$ and a definition of $p$-adic automorphic forms then is there a $p$-adic Langlands transfer?

This paper studies such a transfer. As it seems hard to directly compare spaces of $p$-adic forms on different reductive groups, our approach is to interpolate the classical transfer. For this reason we work with $p$-adic families of automorphic forms, in which the points corresponding to classical forms are sufficiently dense. Such $p$-adic families are provided by eigenvarieties -- these are finite dimensional rigid analytic spaces that interpolate systems of Hecke eigenvalues attached to automorphic representations. They have been constructed for many reductive groups (cf.\ \cite{kevin}, \cite{Chenevier}, \cite{emerton}, \cite{david} and \cite{urban}).
Given a construction of eigenvarieties for two groups $G$ and $H$ as above, one can ask, whether the classical transfer interpolates to a morphism between these rigid spaces.  Calling such a morphism a $p$-adic transfer is justified after one checks that it is compatible with the morphism between the Hecke algebras involved. If we prove the existence of a $p$-adic transfer, then, not only can we transfer points on eigenvarieties that correspond to non-classical points. A $p$-adic transfer gives us a way of comparing $p$-adic families of automorphic forms for different groups in geometric terms.

In this paper we prove the transfer from the group of units in a definite quaternion algebra to the subgroup of norm one elements. 
For the group of norm one elements in a quaternion algebra stable conjugacy and conjugacy do not coincide. As a consequence the classical transfer is proved on the level of $L$-packets rather than representations. It is particularly interesting to study $p$-adic Langlands functoriality in such situations: Issues like endoscopy complicate the picture and cause problems when trying to pass from the classical transfer to a transfer of classical points on eigenvarieties. Before we give an overview of the results, we would like to explain the key points:

Given an eigenvariety $\mathcal{E}$ for the first group, the first step in proving $p$-adic functoriality is to find a compatible eigenvariety for the second group by which we mean that we can define a map from classical points of $\mathcal{E}$ to it. 

One of the parameters of an eigenvariety is that of a tame level, i.e.\ a compact open subgroup $K \subset G(\A^p_f)$, if $G$ is the group in question. In order to find a compatible eigenvariety, a key issue is to find a compatible tame level. We need to understand which tame levels are respected by the classical transfer -- and not just with respect to the transfer of packets but of systems of Hecke eigenvalues attached to automorphic representations, as it is those, which the eigenvarieties interpolate. In this paper we formalise a notion of \textit{Langlands compatibility of tame levels} (see Def.\ \ref{GLCdefspecial} and Def.\ \ref{GLCdef}) in a way that should be easily adaptable to other settings. The main tools we use to prove the existence of compatible tame levels are the multiplicity formulas.

After we have defined a map on classical points we prove that it interpolates to a morphism of rigid analytic spaces. For this we use a very general argument of Bella{\"{\i}}che and Chenevier (cf.\ Proposition \ref{interpolation}). 

The first case of $p$-adic functoriality was studied by Chenevier, who proved a $p$-adic Jacquet--Langlands transfer (see \cite{pjl}). Other results have been established e.g.\ in \cite{hansen}, \cite{newton} and \cite{pjw} and by the author in \cite{me}. To our knowledge all known $p$-adic transfers have in common that classically the transfer works on the level of representations rather than packets, i.e.\ either one works with groups where stable conjugacy and conjugacy agree or if not one imposes extra conditions (like in \cite{me} a stability condition) to enforce this. In the present work there is no such restriction. Furthermore our source eigenvariety can have arbitrary tame level. 

\subsection{Overview of the results}
\label{subsec: overview}
Let $B/\Q$ be a quaternion algebra, which is ramified at $\infty$. Let $\widetilde{G}$ be the algebraic group over $\Q$ defined by the units in $B$, so for any $\Q$-algebra $R$, $\widetilde{G}(R) = (B\otimes_\Q R)^*$. Let $G$ be the subgroup of $\widetilde{G}$ of norm one elements. 
We denote by $S_B$ the finite set of places, where $B$ ramifies and fix a prime $p$, which is not in $S_B$.
We fix models of $\widetilde{G}$ and $G$ over $\Z_{S_B}$.

Using Buzzard's machine Loeffler (in \cite{david}) established the existence of eigenvarieties for a wide class of groups and in particular for the groups $\widetilde{G}$ and $G$. 
The $L$-groups of $\widetilde{G}$ and $G$ are $\GL_2(\Qbar)$ and $\PGL_2(\Qbar)$ respectively. There is a natural projection $\GL_2(\Qbar)\rightarrow \PGL_2(\Qbar)$ and the transfer corresponding to this $L$-homomorphism was proved by Labesse and Langlands in \cite{LL}. To an automorphic representation $\widetilde{\pi}=\otimes \widetilde{\pi}_l$ of $\widetilde{G}(\A)$ one can attach an $L$-packet $\Pi(\widetilde{\pi})$, which is a set of irreducible admissible representations of $G(\A)$.
A representation $\pi=\otimes \pi_l$ of $G(\A)$ belongs to $\Pi(\widetilde{\pi})$ if for all places $l$, $\pi_l$ occurs in the restriction $\widetilde{\pi}_l|_{G(\Q_l)}$.
The finite set of representations~$\pi_l$ occurring in the restriction of $\widetilde{\pi}_l$ to $G(\Q_l)$ forms the local $L$-packet $\Pi(\widetilde{\pi}_l)$.
In \cite{LL}, Labesse and Langlands proved formulas for the multiplicity $m(\pi)$ with which a given representation $\pi \in \Pi(\widetilde{\pi})$ occurs in the discrete automorphic spectrum of $G(\A)$. 

In order to find a suitable target eigenvariety for our $p$-adic transfer, we will first find compatible tame levels in the sense of the next definition. We always assume our tame levels are given as a product of compact open subgroups of $\widetilde{G}(\Q_l)$ (resp.\ of $G(\Q_l)$). 

\begin{defn}\label{GLCdefspecial} [cf.\ Definition \ref{GLCdef}.]
Two tame levels $\widetilde{K} \subset \widetilde{G}(\A^p_f)$ and $K \subset G(\A^p_f)$ are called globally Langlands compatible if the following holds: For any discrete automorphic representation $\widetilde{\pi}$ of $\widetilde{G}(\A)$, such that $(\widetilde{\pi}^p_f)^{\widetilde{K}} \neq 0$ and any $\tau$ in the local $L$-packet $\Pi(\widetilde{\pi}_p)$, there exists an element $\pi$ in the $L$-packet $\Pi(\widetilde{\pi})$ defined by~$\widetilde{\pi}$, such that 
\begin{itemize}
	\item $m(\pi) > 0$,
  \item  $ (\pi^p_f)^K \neq 0$ and
	\item  $\pi_p = \tau$.
\end{itemize}
\end{defn}

In Proposition \ref{gLC} we show that for any tame level $\widetilde{K}\subset \widetilde{G}(\A^p_f)$ we can find a tame level $K\subset G(\A^p_f)$ such that $\widetilde{K}$ and $K$ are globally Langlands compatible. We also show that we can assume that $K_l = \SL_2(\Z_l)$, whenever $\widetilde{K}_l =\GL_2(\Z_l)$. Fix such a pair $\widetilde{K}, K$ of compatible tame levels. 
The compatibility guarantees that we can define a map from a sufficiently dense set $Z$ of classical points of $\widetilde{\D}:=\widetilde{\D}(\widetilde{K})$, the $p$-adic eigenvariety of tame level $\widetilde{K}$, to the set of classical points of the $p$-adic eigenvariety $\D:= \D(K)$. 

Part of the data of an eigenvariety is that of a Hecke algebra, which is in our case a product of local spherical Hecke algebras away from a finite set of primes $S$ and a commutative subalgebra of the Iwahori Hecke algebra at $p$. If $\mathcal{H}_S$ (resp.\ ~$\widetilde{\mathcal{H}}_S$) denotes the Hecke algebra used to build $\D$ (resp.\ $\widetilde{\D}$), then there are structural morphisms $\psi:\mathcal{H}_S \rightarrow \mathcal{O}(\D)$ and $\widetilde{\psi}:\widetilde{\mathcal{H}_S} \rightarrow \mathcal{O}({\widetilde{\D}})$, so we may view the Hecke operators as functions on the eigenvarieties. Furthermore there is a natural injection $\lambda: \mathcal{H}_S\hookrightarrow \widetilde{\mathcal{H}}_S$ (see Section \ref{subsec: iwahori}).
The eigenvariety $\widetilde{\D}$ comes equipped with a morphism $\widetilde{\D} \rightarrow \widetilde{\mathcal{W}}$ to the so called weight space $\widetilde{\mathcal{W}}:= \Hom((\Z_p^*)^2, \mathbb{G}_m)$ and $\D$ lives over the weight space $\mathcal{W}:=\Hom(\Z_p^*, \mathbb{G}_m)$. There is a natural map $\mu:\widetilde{\mathcal{W}} \rightarrow \mathcal{W}.$

A  morphism $\zeta:\widetilde{\D} \rightarrow \D$ is called a $p$-adic Langlands functorial transfer if the diagrams	
$$\xymatrix{
\widetilde{\D}  \ar[d]^{\widetilde{\omega}} \ar[r]^\zeta &\D \ar[d]^{\omega} \\
\widetilde{\mathcal{W}} \ar[r]^\mu &\mathcal{W} }
	\hspace{1cm}
	\xymatrix{
\mathcal{H}_S \ \ar[d]^\psi \ar@{^{(}->}[r]^\lambda & \widetilde{\mathcal{H}}_S \ar[d]^{\widetilde{\psi}}\\
  \mathcal{O}(\D) \ar[r]^{\zeta^*}  &\mathcal{O}(\widetilde{\D}) }  $$ 
  commute.	

Our main theorem is the following.

\begin{thmn}[Theorem \ref{p-adictransfer}]
There exists a $p$-adic Langlands transfer $\zeta:\widetilde{\D} \rightarrow \D$. It has the additional property that it sends a classical point $z \in Z$ to a classical point in $\D$.
\end{thmn}
  
In the proof of the theorem we use two auxiliary eigenvarieties $\D'$ and $\D''$ and ultimately establish the existence of a $p$-adic functoriality morphism $\zeta$ as a composition, as shown in the following commutative diagram
$$\xymatrix{
\widetilde{\D} \ar[d]^{\widetilde{\omega}} \ar[r]^{\zeta'} & \D' \ar[d]^{\omega'} \ar[r]^{\xi} & \ar[d]^{\omega''} \D'' \ar[r] &\D \ar[d]^{\omega} \\
\widetilde{\mathcal{W}} \ar[r]^{\operatorname{id}} & \widetilde{\mathcal{W}} \ar[r]^{\operatorname{id}}  &\widetilde{\mathcal{W}}  \ar[r]^\mu &\mathcal{W} }.$$ 

Here $\D''$ is defined as the pullback of $\D$ along $\widetilde{\mathcal{W}}\rightarrow \mathcal{W}$.
The hard part is establishing the existence of the morphism $\xi$ in the middle, which is where we use the results on the classical transfer.\smallskip\\

There are some parallels of our construction with the classical transfer, which we would like to comment on. In the classical setting starting from an automorphic representation  $\widetilde{\pi}$ of $\widetilde{G}(\A)$ one first passes to an $L$-packet by restricting the local representations $\widetilde{\pi}_l$ to $G(\Q_l)$. Automorphic representations of $\widetilde{G}(\A)$ that give rise to the same $L$-packet are twists of each other. In the second step of the transfer one proves the multiplicity formulas. They allow one to understand which elements of the global $L$-packets are automorphic.

The first eigenvariety $\D'$ corresponds to the restriction process. It is built from the same overconvergent forms as $\widetilde{\D}$ but the underlying Hecke algebra $\mathcal{H}$ is one for~$G$ and the points of $\D'$ correspond to the systems of Hecke eigenvalues restricted to~$\mathcal{H}$. The images of points arising from an automorphic representation $\widetilde{\pi}$ and a twist of it by a Gr{\"o}{\ss}encharacter trivial at infinity agree under the morphism $\zeta'$ from $\widetilde{\D}$ to~$\D'$. We use the family of Galois representations on $\widetilde{\D}$ to study more generally when two points are identified under $\zeta'$ (cf.\ Section \ref{subsubsec: twists}). 

We note however that the choice of our Hecke algebra means that the eigenvariety $\D'$ cannot distinguish between different members of an $L$-packet; the space of automorphic forms that are used in the construction of $\D'$ \textit{sees} different members of an $L$-packet. But $\D'$ itself does not as the $\mathcal{H}$-eigenvalues of any two members of an $L$-packet with correct tame level agree. 

\subsection{Outline of paper}
In Section 2 we collect some background material and establish some auxiliary results. After describing the most important aspects of Buzzard's eigenvariety machine and some key properties of eigenvarieties we recall the general result that we need in the interpolation step. We also establish morphisms between the Hecke algebras and the weight spaces that appear in the construction of the eigenvarieties. The two auxiliary eigenvarieties that we need in the proof of our main theorem are constructed in Section 3.  
In Section 4 we establish the existence of compatible tame levels. Section 5 contains the definition of a $p$-adic functorial transfer and the main theorem.

\subsection{Notation} \label{sec: notation}
For all $p$ fix embeddings $\iota_p:\Qbar \rightarrow \Qbar_p$ and $\iota_\infty:\Qbar \rightarrow \C$. 

Let $X$ be a reduced rigid analytic space. We equip the ring of global analytic functions $\mathcal{O}(X)$ with the coarsest topology such that for any affinoid $Y\subset X$ the induced map $\mathcal{O}(X) \rightarrow \mathcal{O}(Y)$ is continuous. Here $\mathcal{O}(Y)$ carries the Banach-space topology induced from the supremum norm. 

Let $x$ be a point on $X$. A set $\mathcal{V}:=\{V_i: i \in I \} $ of affinoid open neighbourhoods~$V_i$ of $x$ is called a \textit{basis} of affinoid open neighbourhoods of $x$, if for any admissible open $U \subset X$ containing $x$ there exists $i \in I$ such that $V_i \subset U$.  

If $L / \Q_p$ is a finite extension and $X/L$ is a rigid space defined over $L$, then any point $x \in |X|$ is defined over a finite extension $K/L$. We write 
$X(\Qbar_p)$ for the union of the sets $X(K)$ where $K/L$ is finite.

\subsection{Acknowledgments}
The results of this article are part of the author's Ph.D.\ thesis. I would like to thank my supervisor Kevin Buzzard for his patience and his constant support.
I am furthermore grateful to Ga{\"e}tan Chenevier, Toby Gee, David Loeffler and Olivier Ta{\"{\i}}bi for many helpful conversations. Finally I would like to thank the anonymous referee for many helpful comments and for suggesting the elegant argument sketched in Remark \ref{elg}.

\section{Preliminaries on eigenvarieties, Hecke algebras and weight spaces}
\label{chap: background}
\subsection{Generalities on eigenvarietes}
\subsubsection{Buzzard's eigenvariety machine}
The construction of eigenvarieties was axiomatized by Buzzard in \cite{kevin}. We recall the aspects of the eigenvariety machine that we need below. 

Eigenvarieties are constructed relative to a base which usually parametrizes weights. In a first step we assume this base is affinoid. We fix a prime $p$ and a finite extension~$E/\Q_p$.

Let $X$ be a reduced affinoid rigid analytic space over $E$, $M$ a Banach $\mathcal{O}(X)$-module satisfying property (Pr) (see Section 2 of \cite{kevin} for the definition) and let~$\mathbf{T}$ be a commutative $E$-algebra equipped with an $E$-algebra homomorphism $\mathbf{T}\rightarrow \text{End}_{\mathcal{O}(X)}(M)$. Suppose $\phi \in \mathbf{T}$ is an element which acts as a compact endomorphism on $M$. 
\begin{defn} Let $E'$ be a discretely valued extension of $E$. An $E'$-valued system of eigenvalues for $M$ is an $E$-algebra homomorphism $\lambda:\mathbf{T}\rightarrow E'$, such that there exists a point in $X(E')$ (giving a homomorphism $\mathcal{O}(X)\rightarrow E'$) and a nonzero vector $m\in M\widehat{\otimes}_{\mathcal{O}(X)} E'$ with the property that $tm=\lambda(t)m$ for all $t \in \mathbf{T}$. A system of eigenvalues $\lambda$ is called $\phi$-finite if $\lambda(\phi) \neq 0$. 
\end{defn}

\begin{prop}
To the data $(X,M,\mathbf{T},\phi)$ we can associate a reduced separated rigid space $\mathcal{E}_X$ over $E$ endowed with an $E$-algebra homomorphism $\psi: \mathbf{T}\rightarrow \mathcal{O}(\mathcal{E}_X)$ and a morphism of rigid spaces $\omega:\mathcal{E}_X\rightarrow X$ such that 
\begin{enumerate}[(i)]
	\item the morphism $\nu:=(\omega, \psi(\phi)^{-1}):\mathcal{E}_X \rightarrow X\times \mathbb{A}^1$ is finite,
	\item for any open affinoid $V \subset X \times \mathbb{A}^1$, the natural map
$$ \psi \otimes \nu^*: \mathbf{T}\otimes_{E} \mathcal{O}(V) \rightarrow \mathcal{O}(\nu^{-1}(V))$$ 
is surjective.
\item For any discretely valued extension $E'/E$, the map 
$$ \mathcal{E}_X(E')\rightarrow \Hom_{ring}(\mathbf{T}, E') \times X(E'), x \mapsto (\psi_x: h \mapsto \psi(h)(x), \omega(x)) $$
is a bijection between the $E'$-valued points of $\mathcal{E}_X$ and the set of $\phi$-finite $E'$-valued systems of eigenvalues for $M$.
\end{enumerate}
$\mathcal{E}_X$ is uniquely determined by these conditions and is called the eigenvariety of $(X,M,$ $\mathbf{T},\phi)$.
\label{EV}
\end{prop}
\begin{proof} The result is contained in Chapter 5 of \cite{kevin} once we pass to the reduced space of the space constructed there. The finiteness of the morphism in $(i)$ is Lemma 5.3 in loc.\ cit., (ii) follows from the construction,  and (iii) is Lemma 5.9 in \cite{kevin}. The uniqueness is proved as in Proposition 7.2.8 of \cite{BC}.
\end{proof}

\begin{rem} \label{admcover}
Below we need the following details of the construction. Let $Z_\phi \subset X\times \A^1$ be the spectral variety of the compact operator $\phi$ and let $f:Z_\phi \rightarrow X$ denote the canonical projection. In Section 4 of \cite{kevin}, Buzzard constructs an admissible cover~$\mathcal{C}$ of $Z_\phi$. $\mathcal{C}$ consists of affinoid subdomains $Y\subset Z_\phi$ satisfying the following properties: There exists $U \subset X$ affinoid open such that $Y \subset f^{-1}(U)$, the induced map $f:Y\rightarrow U$ is finite and surjective, and $Y$ is disconnected from its complement in~$f^{-1}(U)$. The eigenvariety $\mathcal{E}_X$ is constructed locally over elements of $\mathcal{C}$. If for $Y\in \mathcal{C}$ we denote by $\mathcal{E}(Y)$ the piece of the eigenvariety above $Y$, then $\mathcal{E}_X$ is admissibly covered by the~$\mathcal{E}(Y)$, which is an affinoid $\operatorname{Sp}(\mathbf{T}(Y))$ constructed as follows.

Assume $U=f(Y)$ is connected, then $\mathbf{T}(Y)$ is $(\mathbf{T}(Y)^0)^{red}$, where $\mathbf{T}(Y)^0$ is the image of $\mathbf{T}\otimes_E \mathcal{O}(U)$ in the endomorphism ring $\operatorname{End}_{\mathcal{O}(U)}((M\widehat{\otimes}_{\mathcal{O}(X)}\mathcal{O}(U))_{fs})$. Here we write $fs$ for the finite slope part of $M\widehat{\otimes}_{\mathcal{O}(X)}\mathcal{O}(U)$ by which we mean the module denoted by $N$ in Theorem 3.3 and Section 5 of \cite{kevin}. It is a direct summand of $M\widehat{\otimes}_{\mathcal{O}(X)}\mathcal{O}(U)$, projective of finite rank, invariant under the action of $\mathbf{T}$ and $\phi$ acts invertibly on it. We refer to Sections 3 and 5 of \cite{kevin} for the details regarding the construction. 

For general $Y\in \mathcal{C}$, the image $f(Y)$ might not be connected but we can write $Y$ as a finite disjoint union $Y=\bigsqcup Y_i$, with $Y_i$ such that $f(Y_i)$ is connected, then $\mathcal{E}(Y)$ is defined to be the disjoint union of the $\mathcal{E}(Y_i)$. 
\end{rem}

As in our applications the base space is not necessarily affinoid, we briefly recall the existence theorem of a global eigenvariety. 
For that let $\mathcal{W}$ be a reduced separated rigid analytic space defined over $E$ and as before let $\mathbf{T}$ be a commutative $E$-algebra and $\phi \in \mathbf{T}$ a fixed element. Assume for any affinoid open $X \subset \mathcal{W}$ we are given a Banach $\mathcal{O}(X)$-module $M_X$ as above, i.e. satisfying property (Pr) together with an $E$-algebra homomorphism $\mathbf{T}\rightarrow \text{End}_{\mathcal{O}(X)}(M_X)$ denoted $t\mapsto t_X$, such that $\phi_X$ is compact. Assume furthermore that if $Y\subset X \subset \mathcal{W}$ are two admissible affinoid opens we have a link $\alpha_{XZ}:M_Y \rightarrow M_X \widehat{\otimes}_{\mathcal{O}(X)}\mathcal{O}(Y)$ and that these maps satisfy the compatibility condition $\alpha_{XZ}=\alpha_{XY}\circ \alpha_{YZ}$ whenever $X,Y,Z$ are affinoids with $Z\subset Y\subset X\subset \mathcal{W}$ (for the definition of link see Section 5 of \cite{kevin}). 
We abbreviate this data as the tuple $(\mathcal{W}, M_\bullet, \alpha_\bullet, \mathbf{T}, \phi)$ and call it an eigenvariety datum.
In this global setting we say that a morphism $\lambda: \mathbf{T}\rightarrow E'$ (for $E'/E$ a discretely valued extension) with $\lambda(\phi)\neq 0 $ is a $\phi$-finite $E'$-valued system of eigenvalues for $M_\bullet$ if there exists an affinoid $X \subset \mathcal{W}$, a point in $X(E')$ (and therefore a map $\mathcal{O}(X)\rightarrow E')$ and $0\neq m \in M_X \widehat{\otimes}_{\mathcal{O}(X)} E'$ such that $t\cdot m = \lambda(t)m $ for all $t\in \mathbf{T}$.  

\begin{thm}
To the above data we may canonically associate the eigenvariety $\mathcal{E}$, a reduced separated rigid space over $E$ equipped with an analytic map $\omega:\mathcal{E}\rightarrow \mathcal{W}$, with the property that for any affinoid open $X\subset \mathcal{W}$, the pullback of $\mathcal{E}$ to $X$ is canonically isomorphic to the eigenvariety associated to the datum $(X,M_X,\mathbf{T},\phi)$.
The global eigenvariety $\mathcal{E}$ comes equipped with a ring homomorphism $\psi:\mathbf{T} \rightarrow \mathcal{O}(\mathcal{E})$ such that the following conditions are satisfied:
\begin{enumerate}[(i)]
	\item The morphism $\nu:=(\omega, \psi(\phi)^{-1}):\mathcal{E} \rightarrow \mathcal{W}\times \mathbb{A}^1$ is finite.
  \item For any open affinoid $V \subset \mathcal{W}\times \mathbb{A}^1$, the natural map
$$ \psi \otimes \nu^*: \mathbf{T}\otimes_{E} \mathcal{O}(V) \rightarrow \mathcal{O}(\nu^{-1}(V))$$ 
is surjective.
\item For any discretely valued extension $E'/E$, the map 
$$ \mathcal{E}(E')\rightarrow \Hom_{ring}(\mathbf{T}, E') \times \mathcal{W}(E'), x \mapsto (\psi_x: h \mapsto \psi(h)(x), \omega(x)) $$
is a bijection between the $E'$-valued points of $\mathcal{E}$ and the set of $\phi$-finite $E'$-valued systems of eigenvalues for $M_\bullet$.
\end{enumerate}
Furthermore any point $x \in \mathcal{E}(\Qbar_p)$ has a basis of open affinoid neighbourhoods $V$ of~$x$ such that $\omega(V) \subset \mathcal{W}$ is open affinoid, the morphism $\omega|_V:V\rightarrow \omega(V)$ is finite, and surjective when restricted to any irreducible component of $V$.  
\label{gEV}
\end{thm}

\begin{proof} The existence is Construction 5.7 of \cite{kevin} and (i) follows immediately from it and Theorem \ref{EV}(i). Point (ii) follows from the construction and (iii) is Lemma 5.9 in \cite{kevin}. For the proof of the last paragraph we refer to the proof of Theorem 3.1.1 of \cite{taibi}. 
\end{proof}

\begin{lem}[cf.\ Lemma 7.2.7 of \cite{BC}]\label{7.2.7}
Let $\mathcal{E}$ be an eigenvariety.
\begin{enumerate}[(a)]
\item $\mathcal{E}$ is an admissible increasing union of open affinoids of the form $\nu^{-1}(V)$ for $V \subset \mathcal{W}\times \A^1$ open affinoid. In particular, any two closed points of $\mathcal{E}$ lie in such an open affinoid.
\item For any $x,y \in \mathcal{E}(\overline{\Q}_p), x=y $ if and only if $\psi_x=\psi_y$ and $\omega(x)=\omega(y)$.
\end{enumerate}
\end{lem}

\begin{defn} Let $X$ be a reduced rigid analytic space.
\begin{enumerate}[(i)]
	\item A subset $Z \subset |X|$ is called Zariski-dense if the only analytic subset of $X$ which contains $Z$ is $X$ itself.
	\item We say a subset $Z\subset |X|$ accumulates at $x \in |X|$ if there exists a basis of affinoid open neighbourhoods $U$ of $x$ such that $Z \cap |U|$ is Zariski-dense in $U$. A subset $Z\subset |X|$ is called an accumulation subset if $Z$ accumulates at any $z \in Z$.  
\end{enumerate}
\end{defn}

\begin{lem}
Let $X$ be a rigid analytic space and $T$ an irreducible component of $X$. Assume $Z$ is an accumulation subset and that $ Z\cap |T|$ is non-empty. Then $Z \cap |T| $ is Zariski-dense in $T$.
\label{2.6}
\end{lem}
\begin{proof} By assumption there exists $z \in Z \cap |T|$ and an affinoid neighbourhood $U$ of $z$ in $X$ such that $Z \cap |U| $ is Zariski-dense in $U$. By Corollary 2.2.9 of \cite{conrad}, $T\cap U$ is a union of irreducible components of $U$. $U=\operatorname{Sp}(A)$ is affinoid and so $A$ is noetherian. Therefore any Zariski-dense subset of $U$ is Zariski-dense in each irreducible component of $U$. So $Z\cap |T|\cap |U| $ is Zariski-dense in $T\cap U$. This shows the lemma as the points of any affinoid in an irreducible rigid space are Zariski-dense in it. 
\end{proof}

\begin{defn}\label{evforz} Assume that $(\mathcal{W}, M_\bullet, \alpha_\bullet, \mathbf{T}, \phi)$ is an eigenvariety datum and let $\mathcal{E}$ be the eigenvariety associated to it. Let $\mathcal{Z}\subset \Hom_{ring}(\mathbf{T},\Qbar_p) \times \mathcal{W}(\Qbar_p)$ be any subset. We say $\mathcal{E}$ is an eigenvariety for $\mathcal{Z}$ if there exists an accumulation and Zariski-dense subset $Z \subset \mathcal{E}(\Qbar_p)$ such that the natural evaluation map $\mathcal{E}(\Qbar_p)\rightarrow \Hom_{ring}(\mathbf{T},\Qbar_p)$, $ x \mapsto \psi_x: h \mapsto \psi(h)(x)$
induces a bijection $$Z\stackrel{\sim}{\longrightarrow} \mathcal{Z}, z \mapsto (\psi_z, \omega(z)).$$
\end{defn}

We now use the arguments of \cite{BC} to show how to extend certain maps defined on an accumulation and Zariski-dense subset to morphisms between eigenvarieties. 
Let~$\mathcal{E}$ be an eigenvariety. For any admissible open $V \subset \mathcal{W}\times \mathbb{A}^1$ set $\mathcal{E}_V:= \nu^{-1} (V)$ and let~$\A_V$ denote the affine line over $V$. For each finite set $I \subset \mathbf{T}$ there is a natural morphism defined over $V$
$$ f_{V,I}:\mathcal{E}_V \rightarrow \A_V^{I}, x \mapsto (\psi(h)(x))_{h\in I}.$$
The morphisms commute with any base change by an open immersion $V'\subset V$. Furthermore the $f_{V,I}$ are finite as the morphism $\A_V^{I}\rightarrow V $ is separated and $\nu$ is finite. 
\begin{lem}\label{technical}
Assume that $V$ is affinoid. Then there exists a finite set $I_V \subset \mathbf{T}$ such that for any finite subset $I \subset \mathbf{T}$ with $I\supset I_V$ the morphism $f_{V,I}$ is a closed immersion. 
\end{lem}
\begin{proof} For this let $X:= \mathcal{E}_V$. $\mathcal{O}(X)$ is a finite $\mathcal{O}(V)$-module as $\nu$ is finite. Pick generators $m_1,...,m_l \in \mathcal{O}(X)$. Then by Theorem \ref{gEV} (ii) 	
$$\psi \otimes \nu^*:\mathbf{T}\otimes_E \mathcal{O}(V) \rightarrow \mathcal{O}(X) $$ is surjective and we can write 
$$ m_j = \psi \otimes \nu^*\left(\sum_{k=1}^{n_j} h_{jk} \otimes \lambda_{jk} \right).$$
Choose $I_V= \{h_{jk} \in \mathbf{T}: j=1,...,l \ , k=1,...,n_j \}$.  
We show that $f:=f_{V,I_V}$ is a closed immersion by showing that there exists an admissible affinoid cover $\{U_i\}_{i \in J}$ of~$\A^{I_V}_V$ such that for all $ i\in J$, $f^{-1}(U_i)$ is affinoid and $\mathcal{O}(U_i)\rightarrow \mathcal{O}(f^{-1}(U_i))$ is surjective.
 
Let $|\cdot|$ denote the norm on the affinoid algebra $\mathcal{O}(X)$. Pick $C \in \R$ such that 
$$ |\psi(h_{jk}) | \leq C \ \text{ for all } h_{jk} \in I_V . $$
Let $\varpi$ be a uniformizer of $E$. For all $N\in \Z$ such that $C < |\varpi |^{-N} $ we have 
 $$Im(f) \subset B^{I_V}_V(0,\varpi^{-N}):= \operatorname{Sp}\left(\mathcal{O}(V) \otimes_E E \left\langle \varpi^N T_1,...,\varpi^N T_m\right\rangle\right),$$ 
  where $m= |I_V|$ and furthermore the map on rings 
$$ \mathcal{O}(V) \otimes_E E \left\langle \varpi^N T_1,...,\varpi^N T_m\right\rangle \rightarrow \mathcal{O}(X)$$
 is surjective as its image contains the generators $m_1,...,m_l$. The discs $B^{I_V}_V(0,\varpi^{-N})$ such that $C < |\varpi |^{-N} $ admissibly cover $\A^{|I_V|}_V$ and so we see that $f$ is indeed a closed immersion. It is clear that for all $I\supset I_V$, $f_{I,V}$ is also a closed immersion. 
\end{proof}

The following proposition is the crucial ingredient in the interpolation step of our proof of the main theorem. 

\begin{prop}\label{interpolation}
Let $(\mathcal{W}, M_{1,\bullet}, \alpha_{1,\bullet}, \mathbf{T}, \phi)$ and $(\mathcal{W}, M_{2,\bullet}, \alpha_{2,\bullet}, \mathbf{T}, \phi)$ be two eigenvariety data giving rise to eigenvarieties $\mathcal{E}_1$ and $\mathcal{E}_2$ equipped with morphisms $\psi_i: \mathbf{T} \rightarrow \mathcal{O}(\mathcal{E}_i)$ and $\omega_i:\mathcal{E}_i \rightarrow \mathcal{W}$ for $i=1,2$. Let $Z\subset \mathcal{E}_1(\Qbar_p)$ be an accumulation and Zariski-dense set of points. 

Assume we have an inclusion $ \zeta: Z \hookrightarrow \mathcal{E}_2(\Qbar_p)$ compatible with the maps $\omega_i$ to~$\mathcal{W}$ and such that $\psi_{1,z} = \psi_{2,\zeta(z)} $ for all $z \in Z$. Then $\zeta$ extends to a unique closed immersion $\zeta:\mathcal{E}_1\hookrightarrow \mathcal{E}_2,$ such that the following diagrams commute: 
$$\xymatrix{
\mathcal{E}_1 \ar[rd]_{\omega_1}\ar[rr]^{\zeta} & & \ar[ld]^{\omega_2}\mathcal{E}_2 \\
&  \mathcal{W}  &  }
\hspace{1cm}\xymatrix{
& \mathbf{T} \ar[ld]_{\psi_2} \ar[rd]^{\psi_1} &\\
\mathcal{O}(\mathcal{E}_2) \ar[rr]_{\zeta^*} & & \mathcal{O}(\mathcal{E}_1) .}
 $$  
\end{prop}
\begin{proof} The arguments are identical to the ones used to prove Proposition 7.2.8 in \cite{BC}.
The uniqueness follows from Lemma \ref{7.2.7} and the fact that the $\mathcal{E}_i$ are reduced and separated.
Let $V \subset \mathcal{W} \times \mathbb{A}^1$ be an open affinoid. For $i=1,2$, let $\mathcal{E}_{i,V}:= \nu_i^{-1}(V)$. By Lemma \ref{technical} there exists a finite set $I\subset \mathbf{T}$, such that the natural maps
$$ f_{i,V,I}:\mathcal{E}_{i,V} \rightarrow \A_V^{I}, x \mapsto (\psi_i(h)(x))_{h\in I}$$  
are both closed immersions.
We claim that $f_{1,V,I}(\mathcal{E}_{1,V}) \subset f_{2,V,I}(\mathcal{E}_{2,V})$. For that let $x \in \mathcal{E}_{1,V}$. If $x \in Z$, then $f_{1,V,I}(x)=f_{2,V,I}(\zeta(x))$ by the hypothesis on $\zeta$. $Z$ is Zariski-dense on each irreducible component of $\mathcal{E}_1$, so there exists $z \in Z$ such that $z$ lies on the same irreducible component of $\mathcal{E}_1$ as $x$. By Lemma \ref{7.2.7}(a) and \cite[Lemma 7.2.9]{BC}, we can find an open affinoid $V'\supset V$ such that $z \in \mathcal{E}_{1,V'}\cap Z$ lies on the same irreducible component $T$ of $\mathcal{E}_{1,V'}$ as $x$. By the accumulation property (see Lemma \ref{2.6}) $Z$ is Zariski-dense in $T$, hence for $I'\supset I\cup I_{V'},$
$$ f_{1,V',I'}(T) \subset f_{2,V',I'}(\mathcal{E}_{2,V'})$$
in particular $f_{1,V,I'}(x) \in f_{2,V,I'}(\mathcal{E}_{2,V})$. By projecting to $\mathbb{A}_V^{|I|}$ we get that the same is true for $I'=I$. 
We can therefore define a closed immersion $\zeta_V: \mathcal{E}_{1,V}\rightarrow  \mathcal{E}_{2,V} $ by setting for $I \supset I_V$ 
$$ \zeta_V:= f_{2,V,I}^{-1}\circ f_{1,V,I}$$
This map does not depend on $I$, which implies the commutativity of the diagram on the right. It extends the map $\zeta: Z \cap \mathcal{E}_{1,V}\rightarrow \mathcal{E}_{2,V}$.

If $U \subset V$ is any affinoid open, then $\zeta_V \times_V U = \zeta_U$, hence the $\zeta_V$ glue to a closed immersion $\zeta:\mathcal{E}_1 \hookrightarrow \mathcal{E}_2$. The commutativity of the diagram on the left is clear as for any $V$, $\zeta_V$ is a morphism over $V$.
\end{proof}

\subsection{Hecke Algebras}
In this section we show the existence of an injective morphism between the Hecke algebras for $G$ and $\widetilde{G}$ that appear in the construction of the eigenvarieties. These algebras are the tensor product of an unramified Hecke algebra and a so called Atkin-Lehner algebra at $p$. We treat the unramified Hecke algebra first. 

\subsubsection{Unramified Hecke algebras}
For a locally compact totally disconnected group~$H$ and a compact open subgroup $K$, let $\mathcal{H}(H,K)$ be the Hecke algebra of locally constant compactly supported $\C$-valued functions on $H$ that are bi-invariant under~$K$.

For any prime $l$, the $L$-homomorphism $\GL_2(\Qbar) \rightarrow \PGL_2(\Qbar)$ naturally induces a map of unramified Hecke algebras $$\mathcal{H}(\SL_2(\Q_l),\SL_2(\Z_l))\rightarrow \mathcal{H}(\GL_2(\Q_l),\GL_2(\Z_l)).$$

In order to see that this morphism can be defined over $\Q$ we construct it directly using the Satake isomorphism.
For that let $\widetilde{T}\subset \GL_2(\Q_l)$ be the torus of diagonal matrices and $\widetilde{T}_0:=\widetilde{T} \cap \GL_2(\Z_l)$. 
Similarly $T \subset \SL_2(\Q_l)$ denotes the torus of diagonal matrices and we define $T_0:=T \cap \SL_2(\Z_l)$.
As $\widetilde{T}$ is abelian we have
$$\mathcal{H}(\widetilde{T},\widetilde{T}_0) \cong \C[\widetilde{T}/\widetilde{T}_0] \cong \C[\Z^2],$$
where the last isomorphism comes from the isomorphism $\Q_l^*/\Z_l^*\cong \Z$ induced from the valuation $v:\Q_l^*\rightarrow \Z, v(l)=1$.
We identify $\C[\Z^2]$ with $\C[X,Y,X^{-1},Y^{-1}]$ by sending $(1,0) \in \Z^2$ to $X$ and $(0,1)$ to $Y$ and
denote the resulting isomorphism $\mathcal{H}(\widetilde{T},\widetilde{T}_0) \rightarrow \C[X,Y,X^{-1},Y^{-1}]$ by $\widetilde{\alpha}$.
Similarly $\alpha: \mathcal{H}(T,T_0) \rightarrow \C[Z,Z^{-1}]$ is defined as the composite of 
$$\mathcal{H}(T,T_0) \cong \C[T/T_0] \cong \C[\Z] \cong \C[Z,Z^{-1}].$$
The inclusion $T\subset \widetilde{T}$ induces the homomorphism $\Z \hookrightarrow \Z^2, n \mapsto (n,-n)$, and therefore a map of $\C$-algebras
\begin{eqnarray*}\psi: \C[Z,Z^{-1}] & \hookrightarrow & \C[X,Y,X^{-1},Y^{-1}]\\
Z & \mapsto & XY^{-1}.
\end{eqnarray*}
The morphism $\psi$ restricts to a morphism of invariants under the Weyl group $W \cong \Z/2\Z $ of the tori $\widetilde{T}$ and $T$: 
$$\psi: \C[Z,Z^{-1}]^W \hookrightarrow \C[X,Y,X^{-1},Y^{-1}]^W. $$

Define a monomorphism of $\C$-algebras  
$$\lambda_l: \mathcal{H}(\SL_2(\Q_l),\SL_2(\Z_l))\rightarrow \mathcal{H}(\GL_2(\Q_l),\GL_2(\Z_l))$$
 $$\lambda_l:=(\widetilde{\alpha} \circ \widetilde{\mathcal{S}})^{-1} \circ \psi \circ \alpha \circ \mathcal{S}.$$
Here $\mathcal{S}$ is the Satake isomorphism for $\SL_2(\Q_l)$
\begin{equation*}
\begin{aligned}
\mathcal{S}: \mathcal{H}(\SL_2(\Q_l),\SL_2(\Z_l)) \stackrel{\sim}{\longrightarrow} \mathcal{H}(T,T_0)^W \\
\mathcal{S}f(a)= \delta^{1/2}(a)\int_N{}f(an)dn , 
\end{aligned}
\end{equation*}
where $N$ is the subgroup of unipotent matrices in $\SL_2(\Q_l)$ and $\delta$ is the modulus character. 
Furthermore $\widetilde{\mathcal{S}}: \mathcal{H}(\GL_2(\Q_l),\GL_2(\Z_l)) \stackrel{\sim}{\longrightarrow} \mathcal{H}(\widetilde{T},\widetilde{T}_0)^W$ is the Satake isomorphism for $\GL_2(\Q_l)$ and is given by the same formula.

Let $\mathcal{H}_\Q(\GL_2(\Q_l),\GL_2(\Z_l))\subset \mathcal{H}(\GL_2(\Q_l),\GL_2(\Z_l))$ be the $\Q$-subalgebra of $\Q$-valued functions. This defines a $\Q$-structure on $\mathcal{H}(\GL_2(\Q_l),\GL_2(\Z_l))$. Define the $\Q$-algebra $\mathcal{H}_\Q(\SL_2(\Q_l),\SL_2(\Z_l))$ analogously.

\begin{lem} The morphism $\lambda_l$ restricts to a monomorphism of $\Q$-algebras
$$\lambda_l: \mathcal{H}_{\Q}(\SL_2(\Q_l),\SL_2(\Z_l))\rightarrow \mathcal{H}_{\Q}(\GL_2(\Q_l),\GL_2(\Z_l)),$$
which we again denote by $\lambda_l$.
\end{lem}
\begin{proof}
As $\delta(T)\subset (\Q^*)^2$ the morphism $\alpha \circ \mathcal{S}$ restricts to an isomorphism 
$$(\alpha \circ \mathcal{S})_\Q: \mathcal{H}_\Q(\SL_2(\Q_l),\SL_2(\Z_l)) \stackrel{\sim}{\longrightarrow} \Q[Z,Z^{-1}]^W $$
such that $(\alpha \circ \mathcal{S})_\Q \otimes_\Q \C = \alpha \circ \mathcal{S}$.\footnote{Note that the situation for $\GL_2$ is different. The restriction of $\widetilde{\mathcal{S}}$ to $\mathcal{H}_\Q(\GL_2(\Q_l),\GL_2(\Z_l))$ does not have image in 
$\Q[X,Y,X^{-1},Y^{-1}]^W$ but rather in $\Q(\sqrt{l})[X,Y,X^{-1},Y^{-1}]^W$.} 
The algebra $\mathcal{H}_\Q(\SL_2(\Q_l),\SL_2(\Z_l))$ can be generated by one element, for example by the preimage of $l(Z+Z^{-1})$ under $\alpha \circ \mathcal{S}$. One easily checks that $\psi(l(Z+Z^{-1}))= l(XY^{-1}+X^{-1}Y)$ is in the image of $(\widetilde{\alpha} \circ \widetilde{\mathcal{S}})_\Q$. 
Namely $\sqrt{l}(X+Y)$ is the image of $T_l:= \mathbf{1}_{\GL_2(\Z_l)(^l_0 \ ^0_1)\GL_2(\Z_l)}$, the element $XY$ is the image of $S_l:=\mathbf{1}_{\GL_2(\Z_l)(^l \ _l)\GL_2(\Z_l)}$ and $X^{-1}Y^{-1}$ is the image of $\mathbf{1}_{\GL_2(\Z_l)(^{l^{-1}}\ _{l^{-1}})\GL_2(\Z_l)}$. Therefore 
\begin{equation}
l(XY^{-1}+ X^{-1}Y)= ((\sqrt{l}(X+Y))^2-2lXY)(XY)^{-1} \in Im((\widetilde{\alpha} \circ \widetilde{\mathcal{S}})_\Q).
\label{urops}
\end{equation}
\end{proof}
\begin{rem}\label{elg} We have decided to present this explicit construction in detail here, as it is helpful when one wants to write down concrete elements in the image of~$\lambda_l$. Alternatively and more elegantly one can construct $\lambda_l$ as follows. For a complex reductive group $H$ let $\mathcal{R}(H)$ denote the $\Q$-valued representation ring. The Satake isomorphism gives natural $\Q$-algebra isomorphisms 
\[\mathcal{H}_{\Q}(\SL_2(\Q_l),\SL_2(\Z_l)) \rightarrow \mathcal{R}(\PGL_2), \text{ and}\] 
\[\mathcal{H}_{\Q}(\GL_2(\Q_l),\GL_2(\Z_l))\otimes_\Q \Q(\sqrt{l}) \rightarrow \mathcal{R}(\GL_2)\otimes_{\Q}\Q(\sqrt{l}). \] 
The map $\lambda_l$ constructed above is then just the natural inclusion $\mathcal{R}(\PGL_2) \rightarrow \mathcal{R}(\GL_2)$. Over $\Q$, $\mathcal{R}(\PGL_2)$ is generated by a 3-dimensional representation, namely the reduced adjoint representation $V$ of $\GL_2$ (i.e., endomorphisms of trace 0) and using the Satake isomorphism one can check that $V$ corresponds to a rational Hecke operator, i.e., an element in $\mathcal{H}_{\Q}(\GL_2(\Q_l),\GL_2(\Z_l))$. 
\end{rem}

Let $p$ be a prime different from $l$ and $E$ be a finite extension of $\Q_p$. 
We define 
$$ \lambda_{l,E}:= \lambda_l \otimes id : \mathcal{H}_\Q(\SL_2(\Q_l),\SL_2(\Z_l))\otimes_\Q E \rightarrow \mathcal{H}_\Q(\GL_2(\Q_l),\GL_2(\Z_l)) \otimes_\Q E.$$
If $E$ is clear from the context, we will denote $\lambda_{l,E}$ by $\lambda_l$.
We abbreviate $$\mathcal{H}_E(\SL_2(\Q_l),\SL_2(\Z_l)):=\mathcal{H}_\Q(\SL_2(\Q_l),\SL_2(\Z_l))\otimes_\Q E \ \ \text{and}$$  
$$\mathcal{H}_E(\GL_2(\Q_l),\GL_2(\Z_l)):=\mathcal{H}_\Q(\GL_2(\Q_l),\GL_2(\Z_l)) \otimes_\Q E.$$

\subsubsection{Iwahori Hecke algebras}
We use the notation $\widetilde{G}, G, S_B$ etc.\ as in Section \ref{subsec: overview} and fix a prime $p \notin S_B$. We denote by $\widetilde{I}$ the standard Iwahori subgroup of $\GL_2(\Q_p)$, i.e.\ the subgroup of matrices that are congruent to an upper triangular matrix $\mod p$, and by $I$ the Iwahori subgroup of $\SL_2(\Q_p)$. Then $\widetilde{I}$ and $I$ are admissible in the sense of Definition 2.2.3 of \cite{david}.
Fix Haar measures on $\GL_2(\Q_p)$ and $\SL_2(\Q_p)$ normalized such that $\mathrm{meas}(I)=\mathrm{meas}(\widetilde{I})=1$.  

Let $\widetilde{T} \subset \GL_2(\Q_p)$ and $T \subset \SL_2(\Q_p)$ be the tori of diagonal matrices. We denote by $\widetilde{\Sigma}^+\subset \widetilde{T}$ (resp. $\widetilde{\Sigma}^{++}$) the monoid of all diagonal matrices of the form 
$\left(\begin{smallmatrix} p^{a_1} & \\ & p^{a_2}\end{smallmatrix}\right)$ with $a_1 \leq a_2$ integers (resp. $a_1 < a_2$) and define $\widetilde{\mathbb{I}}$ to be the monoid generated by~$\widetilde{I} $ and~$\widetilde{\Sigma}^+$. 
Let  $\mathcal{H}^+_p(\widetilde{G})$ be the Hecke algebra associated to $\widetilde{\mathbb{I}}$ as in Proposition 3.2.2 of \cite{david}, i.e.\ $\mathcal{H}^+_p(\widetilde{G})$ is the subalgebra of the full Hecke algebra of $\widetilde{G}$ of compactly supported, locally constant $E$-valued functions on $\widetilde{G}$ with support contained in $\widetilde{\mathbb{I}}$. 
Similarly let 
$$\Sigma^+:=\widetilde{\Sigma}^+\cap \SL_2(\Q_p), \ \Sigma^{++}:=\widetilde{\Sigma}^{++}\cap \SL_2(\Q_p), \ \mathbb{I}:= \widetilde{\mathbb{I}} \cap \SL_2(\Q_p)$$
and $\mathcal{H}^+_p(G)$ be the Hecke algebra associated to $\mathbb{I}$. 
Finally let $e_I:=\mathbf{1}_{I} \in \mathcal{H}^+_p(G)$ and $e_{\widetilde{I}}:= \mathbf{1}_{\widetilde{I}} \in \mathcal{H}^+_p(\widetilde{G})$  be the idempotents attached to $I$ and $\widetilde{I}$.

The so called \textit{Atkin--Lehner algebras} $e_I\mathcal{H}^+_p(G)e_{I} $ and $e_{\widetilde{I}}\mathcal{H}^+_p(\widetilde{G})e_{\widetilde{I}}$ have a very simple structure:  

\begin{lem}[Lemma 3.4.1 of \cite{david}]Let $\mathcal{A}_p(G)$ (resp. $\mathcal{A}_p(\widetilde{G})$) be the monoid algebra $E[\Sigma^+]$ (resp. $E[\widetilde{\Sigma}^+]$). 
For $z\in \Sigma^+$ define $\gamma(z):=|I/(I\cap zIz^{-1})|$. Then the morphism 
$$\mathcal{A}_p(G) \rightarrow e_I\mathcal{H}^+_p(G)e_{I},$$ 
defined as the $E$-linear extension of the map
$$\Sigma^+\rightarrow e_I\mathcal{H}^+_p(G)e_{I}, \ z \mapsto \gamma(z)^{-1} \mathbf{1}_{[\operatorname{IzI}]} $$
is an isomorphism of $E$-algebras. 

Similarly define $\widetilde{\gamma}(z):=|\widetilde{I}/(\widetilde{I}\cap z\widetilde{I}z^{-1})|$ for $z \in \widetilde{\Sigma}^+$. Then the morphism
$$ \mathcal{A}_p(\widetilde{G}) \rightarrow e_{\widetilde{I}}\mathcal{H}^+_p(\widetilde{G})e_{\widetilde{I}} $$
defined as the $E$-linear extension of 
$$\widetilde{\Sigma}^+ \rightarrow e_{\widetilde{I}}\mathcal{H}^+_p(\widetilde{G})e_{\widetilde{I}}, \ z \mapsto \widetilde{\gamma}(z)^{-1} \mathbf{1}_{[\operatorname{\widetilde{I}z\widetilde{I}}]} $$
is an isomorphism of $E$-algebras.
\label{AL}
\end{lem}

As $\Sigma^+ \subset \widetilde{\Sigma}^+$ is a submonoid we get a monomorphism 
$$\lambda_p: \mathcal{A}_p(G)\hookrightarrow \mathcal{A}_p(\widetilde{G})$$
of commutative $E$-algebras. Note that the induced monomorphism
$$e_{I}\mathcal{H}^+(G)e_{I} \hookrightarrow e_{\widetilde{I}}\mathcal{H}^+(\widetilde{G})e_{\widetilde{I}}$$
has the property that for $z \in \Sigma^+$ the characteristic function $\mathbf{1}_{[\operatorname{IzI}]}$ maps to $\mathbf{1}_{[\operatorname{\widetilde{I}z\widetilde{I}}]}$. Indeed one easily checks that $\gamma(z)= \widetilde{\gamma}(z)$ for $z \in \Sigma^+$. 

For a finite set of places $S$ including $p$ and the set $S_B$, we define
$$\widetilde{\mathcal{H}}_{ur,S}:= {\bigotimes}_{l\notin S}'\mathcal{H}_E(\GL_2(\Q_l),\GL_2(\Z_l)) \ \text{ and}$$
$$ \mathcal{H}_{ur,S}:={\bigotimes}_{l\notin S}' \mathcal{H}_E(\SL_2(\Q_l),\SL_2(\Z_l)).$$
We moreover define
$$\widetilde{\mathcal{H}}_S:= \mathcal{A}_p(\widetilde{G}) \otimes_E \widetilde{\mathcal{H}}_{ur,S} \ \text{ and } \ 
\mathcal{H}_S:=\mathcal{A}_p(G)\otimes_E \mathcal{H}_{ur,S} . $$  

Putting together the morphisms defined above we get an inclusion of commutative $E$-algebras
$$\lambda:= \lambda_p \otimes {\bigotimes}' \lambda_l : \mathcal{H}_S \hookrightarrow \widetilde{\mathcal{H}}_S.$$
\begin{rem}
Consider the following elements of $\widetilde{\mathcal{H}}_S$ 
$$ U_p := \mathbf{1}_{\left[\operatorname{\widetilde{I} \left(\begin{smallmatrix} 1& \\ & p \end{smallmatrix}\right) \widetilde{I}}\right]}\otimes 1_{\widetilde{\mathcal{H}}_{ur,S}} \text{, and }  U_{p^2} := \mathbf{1}_{\left[\operatorname{\widetilde{I} \left(\begin{smallmatrix} 1& \\ & p^2 \end{smallmatrix}\right) \widetilde{I}}\right]} \otimes 1_{\widetilde{\mathcal{H}}_{ur,S}}, $$
$$u_0 := \mathbf{1}_{\left[\operatorname{\widetilde{I} \left(\begin{smallmatrix} p^{-1}& \\ & p \end{smallmatrix}\right) \widetilde{I}}\right]} \otimes 1_{\widetilde{\mathcal{H}}_{ur,S}} \text{, and }
 S_p := \mathbf{1}_{\left[\operatorname{\widetilde{I} \left(\begin{smallmatrix} p^{-1}& \\ & p^{-1} \end{smallmatrix}\right) \widetilde{I}}\right]} \otimes 1_{\widetilde{\mathcal{H}}_{ur,S}}. $$
Then Lemma \ref{AL} implies that $S_p \in \widetilde{\mathcal{H}}_S^*$, and with our choice of Haar measures 
$$U_{p^2}=(U_p)^2 \text{ and }  u_0 = S_p \cdot U_{p^2}.$$ 
Note that $u_0 \in Im(\lambda)$.
\label{operators}
\end{rem}
\label{subsec: iwahori}

\subsection{Weight spaces}
\label{sec: weights}
The eigenvarieties that we study below are defined relative to the following spaces:
\begin{prop}[\cite{kevin2} Lemma 2]
\begin{enumerate}[(i)]
	\item There exists a separated rigid analytic space $\mathcal{W} = \Hom(\Z_p^*,\mathbb{G}_m)$ over $\Q_p$ such that for any affinoid $\Q_p$-algebra $R$
$$\mathcal{W}(R) = \Hom_{cont}(\Z_p^*,R^*). $$ We call $\mathcal{W}$ the weight space for $G$. It is the union of $p-1$ open unit disks $B$, if $p\neq 2$ and the union of two open unit disks $B$ if $p=2$. 
\item There exists a separated rigid analytic space $\widetilde{\mathcal{W}} = \Hom((\Z_p^*)^2,\mathbb{G}_m)$ over $\Q_p$ such that for any affinoid $\Q_p$-algebra $R$
$$\widetilde{\mathcal{W}}(R) = \Hom_{cont}((\Z_p^*)^2,R^*). $$ We call $\widetilde{\mathcal{W}}$ the weight space for $\widetilde{G}$. It is the union of $(p-1)^2$ (resp.\ $4$) products $B\times B$ of two open unit disks if $p\neq 2$ (resp.\ if $p=2$).
\end{enumerate}
\end{prop}
Note that the construction of weight spaces given in Section 6.4 of \cite{emerton2} agrees with the above as any continuous character of $\Z_p^*$ is locally analytic. 

For any $\Q_p$-affinoid algebra $R$, the inclusion $\Z_p^*\hookrightarrow (\Z_p^*)^2, x \mapsto (x,x^{-1})$ induces the group homomorphism 
\[\Hom_{cont}((\Z_p^*)^2, R^*) \rightarrow \Hom_{cont}(\Z_p^*, R^*), (\chi_1,\chi_2) \mapsto \chi_1/\chi_2
\]
via restriction. We denote by $\mu$ the resulting morphism  
\[\mu:\widetilde{\mathcal{W}} \rightarrow \mathcal{W}.
\]
Up to a change of coordinates $(x,y)\mapsto (x/y,y)$ this is the projection map to a factor, hence $\mu$ is flat. 

Recall that a rigid space $X$ over $E$ is called \textit{nested} if it has an admissible cover by open affinoids $\{X_i,i\geq0\}$ such that $X_i \subset X_{i+1}$ and the natural $E$-linear map $\mathcal{O}(X_{i+1})\rightarrow \mathcal{O}(X_i)$ is compact. 
Note that both $\mathcal{W}$ and $\widetilde{\mathcal{W}}$ are nested.   
\begin{lem}
Let $V\subset \widetilde{\mathcal{W}}$ be an affinoid open. Then there exists an affinoid $U \subset \mathcal{W}$ such that $\mu(V)\subset U$. 
\label{nested}
\end{lem}
\begin{proof} Let $q=(p-1)$ if $p \neq 2$ and $q=2$ if $p=2$. Then $\widetilde{\mathcal{W}}$ is a disjoint union of $q^2$ products $B\times B$ of two open unit disks and we label them by $B_{j,k}$, $j=1,\ldots,q, k=1,\ldots,q$. By changing coordinates we can assume that on any of these products $\mu$ is the projection onto the first disk. Choose a nested cover $\cup_{i \in \mathbb{N}} X_i$ of $\widetilde{\mathcal{W}}$, where $X_i := \cup_{j,k=1}^q X_{i,j}\times X_{i,k}$  and $X_{i,j}\times X_{i,k} \subset B_{j,k}$ is the product of two closed disks of radius $p^{-1/i}$. Likewise we choose the nested cover $\cup_{i \in \mathbb{N}} U_i$ of $\mathcal{W}$, where $U_i:= \cup_{j=1}^q X_{i,j}$. Then for any $i \in \mathbb{N}$ and $ j,k \in \{1,\ldots,q\},$ the morphism $\mu|_{X_i}$ factors through $U_i$. So we can view $\mu$ as being glued from the projections $\mu_{i}:X_i \rightarrow U_i$. 
Now let $V\subset \widetilde{\mathcal{W}}$ be an open affinoid. Then there exists $i$ such that $V \subset X_i$ and therefore $\mu(V)= \mu_i(V) \subset U_i$. 
\end{proof}

\begin{rem}
Note that for any $i \in \mathbb{N}$, the morphism $\mu_i$ defined in the proof of the last lemma is a flat morphism of quasi-compact and separated spaces. This implies 
that in fact the image $\mu(V)$ is an admissible open subset of $\mathcal{W}$. 
It is furthermore quasi-compact, as $V$ is quasi-compact (because it is affinoid). 
\end{rem}

In the context of eigenvarieties we will consider the base change $\widetilde{\mathcal{W}}\times_{\operatorname{Sp}(\Q_p)} \operatorname{Sp}(E)$ and $\mathcal{W}\times_{\operatorname{Sp}(\Q_p)} \operatorname{Sp}(E)$ for a finite extension $E/\Q_p$. Note that Lemma \ref{nested} is also true for the base change of $\mu$. In order to keep the notation simple, we will omit the base change from the notation: $\widetilde{\mathcal{W}}$ and $\mathcal{W}$ will denote the base change to a finite extension $E$ of $\Q_p$, which will be clear from the context.

Furthermore we denote by $\Delta: \widetilde{T}_0\rightarrow \mathcal{O}(\widetilde{\mathcal{W}})^* $ the universal character and for any $X \subset \widetilde{\mathcal{W}}$ affinoid open we let 
$$\Delta_X : \widetilde{T}_0 \rightarrow \mathcal{O}(X)^*$$
be the composition of $\Delta$ and the natural homomorphism $\mathcal{O}(\widetilde{\mathcal{W}})\rightarrow \mathcal{O}(X)$.

\section{A zoo of eigenvarieties}
\label{sec: zoo}
Let $B$, $\widetilde{G}$, $S_B$ and $p$ be as above and let $E/\Q_p$ be a finite extension. In this section we briefly recall some details of the construction in \cite{david}, check some technical properties of the eigenvarieties and construct the two auxiliary eigenvarieties that we need for the $p$-adic transfer. 
We remark here that due to the generality Loeffler works in he has to impose certain \textit{arithmetical} conditions at various places. In our case however one easily checks that $\widetilde{G}$ and $G$ satisfy the conditions of \cite{gross} Proposition 1.4.\footnote{E.g.\ condition (5) is satisfied as the centres of $\widetilde{G}$ and of $G$ are maximal split tori.} All arithmetic subgroups of $\widetilde{G}$ and $G$ are therefore finite, which implies that we do not have to worry about any of the arithmetical conditions.  

\subsection{Eigenvarieties of idempotent type for $\widetilde{G}$}
\label{subsec: firsteigenvarieties}
\label{sec: 3.2.1}
Choose a product Haar measure $\widetilde{\mu}=\prod_l \widetilde{\mu}_l$ on $\widetilde{G}(\A_f)$ such that for all $l \notin S_B \cup \{p\}$, $\widetilde{\mu}_l(\GL_2(\Z_l)) = 1$ and such that $\mu_p(\widetilde{I})=1$.

\subsubsection{Review of $p$-adic automorphic forms and eigenvarieties of idempotent type}

Let $X \subset \widetilde{\mathcal{W}}/E$ be an open affinoid and $V$ be a locally $\Q_p$-analytic representation of $\widetilde{T}_0 \cong (\Z^*_p)^2$. Let
$k(X)$ be the minimal integer $k'\geq 0$ such that the restriction of $\Delta_{X}$ to $\widetilde{T}_{k'}:=\{ A \in \widetilde{T}(\Z_p) \ | \  A  \equiv 1 \text{ mod  }p^{k'} \}$ is analytic. Let $k\geq k(X)$ be an integer big enough such that the restriction of $V$ to $\widetilde{T}_k$ is analytic.

In Section 2 of \cite{david} Loeffler constructs orthonormalizable Banach $\mathcal{O}(X)$-modules $\mathcal{C}(X,V,k)$. We assume throughout that $V:= \textbf{1}$ is the trivial representation and write $\mathcal{C}(X,k)$ for $\mathcal{C}(X,1,k)$. The space $\mathcal{C}(X,k)$ is equipped with an action of the monoid $\widetilde{\mathbb{I}}$. 
Denote by $\mathcal{L}(\mathcal{C}(X,k))$ the $\mathcal{O}(X)$-module of $\mathcal{C}(X,k)$-valued automorphic forms (see Def. 3.3.2 of \cite{david}).
Fix an idempotent
$$\widetilde{e}=\prod_{l\neq p}\widetilde{e}_l \in C^\infty_c(\widetilde{G}(\A_f^{p}), E).$$ 
Let $S$ be a finite set of primes containing $p$ and the set $S_B$ and such that $\widetilde{e}_l =\mathbf{1}_{\GL_2(\Z_l)}$ is the identity in $\mathcal{H}_E(\GL_2(\Q_l),\GL_2(\Z_l))$for all $l\notin S$. Define $\widetilde{e}_p:= \widetilde{e}_{\widetilde{I}}$ and
 $$M(\widetilde{e},X,k):= (\widetilde{e}\cdot \widetilde{e}_p) \mathcal{L}(\mathcal{C}(X,k)).$$ 

It is a Banach $\mathcal{O}(X)$-module satisfying property (Pr) and carries an action of the algebra $\widetilde{\mathcal{H}}_{S}$. If $u \in \widetilde{\mathcal{H}}_S$ is supported in $\widetilde{G}(\A^{S})\times \widetilde{I}\widetilde{\Sigma}^{++}\widetilde{I}$ then by Theorem 3.7.2 of \cite{david}, $u$ acts as a compact operator on $M(\widetilde{e},X,k)$. Fix such a $u$.

Finally for $Y\subset X $ an open affinoid we have links 
$$\alpha_{XY}:M(\widetilde{e},Y) \rightarrow M(\widetilde{e},X) \widehat{\otimes}_{\mathcal{O}(X)}\mathcal{O}(Y)$$
for the modules $M(\widetilde{e},X):=M(\widetilde{e},X,k(X))$ as constructed in Lemma 3.12.2 of \cite{david}.

\begin{prop}[cf.\ Theorem 3.11.3 of \cite{david}] The data $(\widetilde{\mathcal{W}}, M_\bullet(\widetilde{e}), \alpha_{\bullet}, \widetilde{\mathcal{H}}_S,u)$ is an eigenvariety datum. We denote by $\widetilde{\mathcal{D}}$ the associated eigenvariety and call it an eigenvariety of idempotent type for $\widetilde{G}$.
\end{prop}
\begin{rem}
We sometimes need to distinguish between different eigenvarieties of idempotent type for $\widetilde{G}$. We then decorate them accordingly as $\widetilde{\mathcal{D}}_{u}$, $\mathcal{D}(\widetilde{e})$ or $\D_u(\widetilde{e})$.
\end{rem}

Recall the definition of the operator $ u_0 $ from Remark \ref{operators}. 
\begin{lem}\label{changeu} Let $u = \mathbf{1}_{\operatorname{[\widetilde{I}z\widetilde{I}]}} \otimes 1_{\widetilde{\mathcal{H}}_{ur,S}}$ and $u' = \mathbf{1}_{\operatorname{[\widetilde{I}z'\widetilde{I}]}} \otimes 1_{\widetilde{\mathcal{H}}_{ur,S}}$ be two elements of~$\widetilde{\mathcal{H}}_S$ such that $z,z' \in \widetilde{\Sigma}^{++}$. Then for any idempotent $\widetilde{e}$ the eigenvarieties 
$\D_u(\widetilde{e})$ and~$\D_{u'}(\widetilde{e})$ are isomorphic. In particular for any such $u$, $\D_u(\widetilde{e})$ is isomorphic to~$\mathcal{D}_{u_0}(\widetilde{e})$.
\end{lem}
\begin{proof} 
By Lemma 3.4.2 of \cite{david} we get that locally over an affinoid $X$ of weight space, the image of $u$ in $\operatorname{End}(M(\widetilde{e},X))$ is invertible if and only if the image of $u'$ is invertible. The lemma now follows from Corollary 3.11.4 of \cite{david}.
\end{proof}

From now on we only consider eigenvarieties $\widetilde{\D}$ that are built with respect to an operator $u$ of the form $\mathbf{1}_{\operatorname{[\widetilde{I}z\widetilde{I}]}} \otimes 1_{\widetilde{\mathcal{H}}_{ur,S}}$ for $z \in \widetilde{\Sigma}^{++}$. 

\subsubsection{Classical points}
We show that, in the sense of Definition \ref{evforz}, the eigenvariety~$\D(e)$ is an eigenvariety for a set $\mathcal{Z}_e$ arising from classical automorphic representations. We choose once and for all a square root of $p$ in $\Qbar$. Define $\widetilde{\Sigma}:= \widetilde{T}/\widetilde{T}_0$ and let $\delta_{\widetilde{B}}:\widetilde{\Sigma} \rightarrow \C^*$ be the modulus character of the Borel subgroup $\widetilde{B} \subset \GL_2(\Q_p)$ of upper triangular matrices, so $\delta_{\widetilde{B}}(p^a, p^b)=|p^a|/|p^b|= p^{b-a}.$ Let $\pi_p$ be an irreducible smooth representation of $\GL_2(\Q_p)$ and assume it has a non-zero fixed vector under~$\widetilde{I}$. 
\begin{defn}[cf.\ \cite{BC} Def.\ 6.4.6]
A character $\chi:\widetilde{\Sigma} \rightarrow \C^*$ is called an accessible refinement of $\pi_p$ if $\chi \delta_{\widetilde{B}}^{-1/2}:\widetilde{\Sigma} \rightarrow \C^*$ occurs in $\pi_p^{\widetilde{I}}$. 
\end{defn}
Using Jacquet-modules and the Geometric Lemma of Bernstein and Zelevinsky one can see that $\chi$ is an accessible refinement if and only if $\pi_p$ is a subrepresentation of $Ind_{\widetilde{B}}^{\GL_2(\Q_p)}(\chi)$ (cf.\ Section 6.4 of \cite{BC}). 

The next definition is analogous to the one made in Section 7.2.2 of \cite{BC} in the context of unitary groups. Let $\underline{k}:=(k_1,k_2) \in \Z^2$ be a pair of integers such that $k_1\geq k_2$.
\begin{defn}A $p$-refined automorphic representation of weight $\underline{k}$ of $\widetilde{G}(\A)$ is a pair $(\pi, \chi)$ such that
\begin{itemize}
	\item $\pi$ is an automorphic representation of $\widetilde{G}(\A)$, 
	\item $\pi_p$ has a non-zero $\widetilde{I}$-fixed vector and $\chi: \widetilde{\Sigma} \rightarrow \C^*$ is an accessible refinement of $\pi_p$, 
	\item $\pi_\infty \cong  (\operatorname{Sym}^{k_1-k_2} (\C^2) \otimes \operatorname{Nrd}^{k_2})^*$. 
\end{itemize}
\label{pref}
\end{defn}
Note that our terminology is non-standard in the sense that $\pi_\infty$ is the dual of the representation one would expect. The reason for our choice of terminology is that with Loeffler's normalizations a $p$-refined automorphic representation gives rise to a point on an eigenvariety, whose image in weight space $\mathcal{W}$ corresponds to $\underline{k}$; and it is this weight we want to keep track of (cf.\ Lemma \ref{clapoint} below).

Let $\widetilde{e}:= \prod{\widetilde{e}_l	} \in C^\infty_c(\widetilde{G}(\A_f^{p}), \Qbar)$ be an idempotent and as before fix a finite set of primes $S$ containing $p$, $S_B$ and all primes $l$ where $\widetilde{e}_l$ is not equal to $\mathbf{1}_{\GL_2(\Z_l)}$.
Assume $(\pi,\chi)$ is a $p$-refined automorphic representation such that $\iota_\infty(\widetilde{e})\pi^{p}_f \neq 0$, so in particular $\pi_l$ is unramified for all $l \notin S$.

We attach to $(\pi,\chi)$ a $\Qbar_p$-valued character $\psi_{(\pi,\chi)}$ of $\widetilde{\mathcal{H}}_S$ as follows: 
The space $(\pi_f^S)^{\widetilde{G}(\widehat{\Z}^S)}$ is 1-dimensional and $\widetilde{\mathcal{H}}_{\text{ur}}$ acts on it via a character, which we denote by~$\chi_{\text{ur}}$. As $\pi$ is defined over $\Qbar$, $\chi_{\text{ur}}$ as well as $\chi$ are valued in $\iota_\infty(\Qbar)$ (in fact in a number field).
We turn $\chi\delta_{\widetilde{B}}^{-1/2}$ into a character of $\mathcal{A}_p(\widetilde{G})$ by restricting $\iota_p \circ \iota_\infty^{-1}\circ(\chi\delta_{\widetilde{B}}^{-1/2})$ to $\widetilde{\Sigma}^+$ and extending it $E$-linearly to $\mathcal{A}_p(\widetilde{G})$. Abusing notation we denote the resulting character again by $\chi\delta_{\widetilde{B}}^{-1/2}$.
Let $\delta_{\underline{k}}:\mathcal{A}_p(\widetilde{G}) \rightarrow \Qbar_p$ be the character given by $\delta_{\underline{k}} (p^{a_1},p^{a_2}) = p^{-k_1 a_1-k_2a_2}$. Then we define
$$\psi_{(\pi,\chi)}: \widetilde{\mathcal{H}}_S \rightarrow \Qbar_p , \ \ \psi_{(\pi,\chi)}:= \chi \delta_{\widetilde{B}}^{-1/2}\delta_{\underline{k}} \otimes \chi_{\text{ur}},$$
and refer to it as the \textit{character of $\widetilde{\mathcal{H}}_S$ associated with $(\pi,\chi)$}.

We embed $\Z^2 \rightarrow \widetilde{\mathcal{W}}(\Qbar_p)$ via $(k_1,k_2) \mapsto ( (z_1,z_2) \mapsto z_1^{k_1}z_2^{k_2})$ and define $\mathcal{Z}_{\widetilde{e}} \subset \Hom_{ring}(\widetilde{\mathcal{H}}_S,\Qbar_p) \times \mathcal{W}(\Qbar_p)$ to be the subset of pairs
\begin{equation}
\begin{aligned} \{(\psi_{(\pi,\chi)}, \underline{k}) \ | &\ \pi \text{ is a $p$-refined automorphic representation of weight } \underline{k} \text{ such that}\\
 &  \pi_p \text{ is an irreducible unramified principal series and } \iota_\infty(\widetilde{e})\pi_f^p \neq 0 \}.
\end{aligned}
\label{cp}
\end{equation}

Let $E$ be a finite extension of $\Q_p$ such that $\iota_p(\widetilde{e})$ has values in $E$ and let $\widetilde{\mathcal{D}}:= \mathcal{D}(\iota_p(\widetilde{e}))$ be the eigenvariety associated to $\iota_p(\widetilde{e})$. Any $(\psi_{(\pi,\chi)}, \underline{k})\in \mathcal{Z}_{\widetilde{e}}$ gives rise to a point in $\widetilde{\D}$. More precisely let $X \subset \widetilde{\mathcal{W}}$ be an open affinoid containing $\underline{k}$ and $E'/ E$ be a finite extension containing the values of $\psi_{(\pi,\chi)}$. We view $\underline{k}$ as an $E'$-valued point of $\widetilde{\mathcal{W}}$, in particular this gives a morphism $\underline{k}: \mathcal{O}(X) \rightarrow E'$.
\begin{lem}
Let $(\psi_{(\pi,\chi)}, \underline{k})\in \mathcal{Z}_{\widetilde{e}}$ be an element. Then there exists $0 \neq v \in M(\widetilde{e},X)\widehat{\otimes}_{\mathcal{O}(X),\underline{k}} E'$ such that $h\cdot v = \psi_{(\pi,\chi)}(h)v$ for all $h\in \widetilde{\mathcal{H}}_S$. 
\label{clapoint}
\end{lem}
\begin{proof}
This follows from the construction (see Section 3.9 of \cite{david}). For further details the reader may consult Section 3.2.1 of \cite{mythesis}. 
\end{proof}
Therefore any $x \in \mathcal{Z}_{\widetilde{e}}$ defines a point in $\widetilde{\mathcal{D}}$. Lemma \ref{7.2.7} implies that we get an injection $ \mathcal{Z}_{\widetilde{e}} \hookrightarrow \widetilde{\D}(\Qbar_p)$ whose image we denote by $Z$. Next we show that $Z$ accumulates in $\widetilde{\D}(\Qbar_p)$ by following arguments analogous to the ones in Chapter 6 of \cite{Chenevier}.

Let $\mathcal{C}(\underline{k},k(X)):= \mathcal{C}(X,k(X))\widehat{\otimes}_{\mathcal{O}(X),\underline{k}} E$ be the fibre of $\mathcal{C}(X,k(X))$ above $\underline{k} \in \widetilde{\mathcal{W}}(E)$. It carries an action of $\widetilde{\mathbb{I}}$.
There is an inclusion of $\widetilde{\mathbb{I}}$-representations
\begin{equation}
W_{\underline{k},E}\otimes \sigma_{\underline{k}} \subset \mathcal{C}(\underline{k},k(X)),
\label{normalizations2}
\end{equation}
where $W_{\underline{k},E}:= \operatorname{Sym}^{k_1-k_2}(E^2)\otimes \operatorname{det}^{k_2}$ is the algebraic representation of weight $\underline{k}$ of the group $\GL_2(\Q_p)$ and $\sigma_{\underline{k}}:\widetilde{\mathbb{I}}\rightarrow E^*$ is the character, which is trivial on $\widetilde{I}$ and equal to $z \mapsto z_1^{-k_1}z_2^{-k_2}$ for all $z \in \widetilde{\Sigma}^+$.\footnote{This is well-defined as for $z_1, z_2 \in \widetilde{\Sigma}^+$ we have $z_1\widetilde{I}z_2\subset \widetilde{I}z_1z_2\widetilde{I}$.} Define $W:=W_{\underline{k},E}\otimes \sigma_{\underline{k}}$.

\begin{defn}Let $M(\widetilde{e},\underline{k}, k(X)):= \widetilde{e} \mathcal{L}(\mathcal{C}(\underline{k},k(X)))$. 
We refer to 
\[M(\widetilde{e},\underline{k})^{cl}:= \widetilde{e} \mathcal{L}(W) \subset M(\widetilde{e},\underline{k},k(X))
\]
as the subspace of classical forms. If $E'/E$ is a finite extension, we say that a point $z \in \widetilde{\D}(E')$ of weight $\omega(z)=\underline{k}$ is classical if there exists $0 \neq v \in M(\widetilde{e},\underline{k})^{cl}\widehat{\otimes}_E E'$ such that $h \cdot v = \psi_z(h) v$ for all $h \in \widetilde{\mathcal{H}}_S$. 
\end{defn}

Below we will make use of the following theorem which shows that forms of small slope are classical. Again let $E'/E$ be a finite extension. 
\begin{thm}
Let $\underline{k}=(k_1,k_2)$ be as above. Let $\lambda \in E'^*$ and $\sigma:= v_p(\lambda)$. If 
$$ \sigma < 2(k_1 - k_2 +1)$$
then the generalized $\lambda$-eigenspace of $u_0$ acting on $M(\widetilde{e},\underline{k},k(X))\widehat{\otimes}_E E'$ is contained in the subspace $M(\widetilde{e},\underline{k})^{cl}\widehat{\otimes}_E E'$.  
\end{thm}
\begin{proof} This is Theorem 3.9.6 of \cite{david}. We have used Proposition 2.6.4 of loc.\ cit.\ to show that if $\sigma < 2(k_1 - k_2 + 1)$ then $\sigma $ is a small slope for $u_0$ in the sense of Definition 2.6.1 of loc.\ cit. 
\end{proof}
If $z \in Z$ is a point corresponding to a $p$-refined automorphic representation $(\pi,\chi)$ of weight $\underline{k}$ then for $\chi=(\chi_1,\chi_2)$ we have
$\psi(u_0)(z) = \iota_p(\chi_2(p) \chi_1(p)^{-1}) p^{k_1-k_2+1} .$
Note that if $z$ is classical and $\psi(u_0)(z) \neq p^{1+ k_1-k_2 \pm 1 }$ then $\pi_p$ is an unramified irreducible principal series and so $z$ belongs to $Z$. 

\begin{prop}
$\widetilde{\mathcal{D}}$ is an eigenvariety for $\mathcal{Z}_{\widetilde{e}}$ in the sense of Definition \ref{evforz}.
\label{cp}
\end{prop}
\begin{proof} We show that $Z$ is accumulation and Zariski-dense using the same arguments as in Section 6.4.5 of \cite{Chenevier}.
For the accumulation property let $z \in Z$ be a point and $z \in U$ an affinoid neighbourhood. Without loss of generality $U$ is such that $\omega(U)\subset \widetilde{\mathcal{W}}$ is open affinoid and connected, the morphism $\omega|_U:U \rightarrow \omega(U)$ is finite, and surjective when restricted to any irreducible component of $U$. 
The function~$\psi(u_0)$ is non-zero on $U$ and is therefore bounded from above and below. Therefore we can find $r \in \R$ such that 
$$ |\psi(u_0)(x)| > r \ \ \forall x \in U(\Qbar_p),$$
 and $\sigma_0 \in \R$
such that 
$$v_p(\psi(u_0)(x)) \leq \sigma_0 , \ \forall x \in U(\Qbar_p).$$
Let $\underline{k}^0=(k^0_1,k^0_2):=\omega(z)$ and $N$ be big enough such that 
$$\{(k_1,k_2) \in \Z^2 \ | \ k_i=k^0_i \text{ mod } p^N(p-1)\} \subset \omega(U).$$
For any $C,\sigma \in \R$ the set
$$\mathcal{X}_{C,\sigma}:=\{(k_1,k_2) \in \Z^2 | k_1-k_2 > max\{0,\sigma/2-1, C\} \}$$
is Zariski-dense in $\mathcal{W}$ and accumulates at all algebraic weights  (cf.\ Lemma 2.7 of \cite{chenevier2}).
Therefore the set
$$\mathcal{X}_{C,\sigma, N}:=\{(k_1,k_2) \in \Z^2 | k_1-k_2 > max\{0,\sigma/2-1, C\}, k_i=k^0_i \text{ mod }p^N (p-1)  \}$$
is Zariski-dense in $\omega(U)$.  
Then by Lemma 6.2.8 of \cite{Chenevier} the preimage 
\[(\omega|_U)^{-1}(\mathcal{X}_{C,N,\sigma})
\]
is Zariski-dense in $U$. 
Setting $C>v_p(r)$ and $\sigma > \sigma_0$ we have that any element in this preimage arises from a point in $Z$. This proves the accumulation property.

For the Zariski-density note that by Lemma \ref{2.6} it suffices to show that each irreducible component contains a point of $Z$, which follows as any irreducible component of $\widetilde{\mathcal{D}}$ surjects onto an irreducible component of $\widetilde{\mathcal{W}}$. 
\end{proof}

\subsubsection{The Galois representation attached to a point $x$ in $\widetilde{\mathcal{D}}(\overline{\Q}_p)$.}
For a finite set $S$ of places of $\Q$ we denote by $G_{\Q,S}$ the Galois group of a maximal algebraic extension of~$\Q$ that is unramified outside $S$. 
An eigenvariety $\widetilde{\D}$ of idempotent type for $\widetilde{G}$ carries a pseudo-representation and in particular we can attach to any point $x \in \widetilde{\D}(\Qbar_p)$ a Galois representation $\rho(x):G_{\Q,S}\rightarrow \GL_2(\Qbar_p)$. This can be proved just as in the case of eigenvarieties for unitary groups by applying the very general results of Section 7.1 of \cite{Chenevier}.\footnote{The fact that all rigid spaces in \cite{Chenevier} are over $\C_p$ whereas our rigid spaces are defined over finite extensions $E$ of $\Q_p$ does not affect the results.} 
We fix an eigenvariety $\widetilde{\D}:=\D_u(\widetilde{e})$ of idempotent type for $\widetilde{G}$, so in particular an idempotent $\widetilde{e}$ and a set $S$ of bad places. We drop the supscript $S$ from the notation in all Hecke algebras in this section as there is no confusion. 

Let $\widetilde{\mathcal{H}}^0_{ur}\subset \widetilde{\mathcal{H}}_{ur}$ be the subalgebra of $\mathcal{O}_E$-valued functions. Then, using Proposition 3.3.3 and Proposition 3.5.2 of \cite{david} as well as the fact that 
\[\mathcal{O}(\widetilde{\D})^0= \{f \in \mathcal{O}(\widetilde{\D}): \ \forall x \in |\widetilde{\D}|, |f(x)|\leq 1 \} 
\]
is compact, we see that the image of $\widetilde{\mathcal{H}}^0_{ur}$ under the map $\psi: \widetilde{\mathcal{H}}\rightarrow \mathcal{O}(\widetilde{\D})$ is relatively compact. 
Define $\mathcal{H}_{\widetilde{\D}}$ to be the closure of $\psi(1_{\operatorname{\widetilde{I}}}\otimes \widetilde{\mathcal{H}}_{\text{ur}}^0)$ in $\mathcal{O}(\widetilde{\D})$.

For $l \notin S$ let $Frob_l \in G_{\Q,S}$ be a Frobenius. 
For $l \notin S$ define the element $T_l := \psi\left(\mathbf{1}_{\operatorname{\widetilde{I}}} \otimes \mathbf{1}_{\GL_2(\widehat{\Z}^S) \left(\begin{smallmatrix} l& \\ & 1 \end{smallmatrix}\right)\GL_2(\widehat{\Z}^S)}\right) \in \mathcal{H}_{\widetilde{\D}}$. Here $\left(\begin{smallmatrix} l& \\ & 1 \end{smallmatrix}\right)$ is understood to be the matrix in $\GL_2(\widehat{\Z}^S)=\prod_{q\notin S} \GL_2(\Z_q)$, which is equal to $1$ for all $q \neq l$ and equal to $\left(\begin{smallmatrix} l& \\ & 1 \end{smallmatrix}\right)$ at $l$. 
Below we also use the elements $S_l := \psi \left(\mathbf{1}_{\operatorname{\widetilde{I}}} \otimes \mathbf{1}_{\GL_2(\widehat{\Z}^S) \left(\begin{smallmatrix} l& \\ & l \end{smallmatrix}\right)\GL_2(\widehat{\Z}^S)}\right) \in \mathcal{H}_{\widetilde{\D}}$ for $l \notin S$.
 
Note that Proposition \ref{cp} implies that Hypothesis \textbf{H} in Section 7.1 of \cite{Chenevier} is satisfied:
There exists a Zariski-dense set of points $Z \subset \widetilde{\D}(\overline{\Q}_p)$ and for each $z \in Z$ a continuous representation 
$$\rho(z): G_{\Q,S} \rightarrow \GL_2(\overline{\Q}_p), $$
such that for all $l \notin S$, $tr(\rho(z)(Frob_l))= T_l(z)$. Furthermore, for each $z \in Z$, $T_z:=tr(\rho(z)(\cdot)):G_{\Q,S} \rightarrow \overline{\Q}_p$ is a pseudo-representation of dimension $2$ of $G_{\Q,S}$ into $\overline{\Q}_p$.
\begin{prop}
There exists a unique continuous pseudo-representation 
$$\mathcal{T}: G_{\Q,S}\rightarrow \mathcal{H}_{\widetilde{\D}}$$
of dimension 2 such that for each $z \in Z$ the evaluation of $\mathcal{T}$ at $z$ is equal to $T_z$. 
\end{prop}
\begin{proof} This is Proposition 7.1.1 of \cite{Chenevier} once we replace $\prod_{z \in Z}\C_p$ by $\prod_{z \in Z} E_z$, where~$E_z$ is the residue field at $z$, in the proof.
\end{proof}

We may use the same arguments to interpolate the one-dimensional pseudo-representations given by the determinant of $\rho(z)$, to build a ``determinant function'' on $\widetilde{\D}$, i.e.\ a function $\mathcal{S}:G_{\Q,S} \rightarrow \mathcal{O}(\widetilde{\D})^*$ such that 
\[\mathcal{S}(Frob_l)(z)=det(\rho(z)(Frob_l))
\]
 for all $l \notin S$ and $z \in Z$.

Recall that for $z \in Z$ the characteristic polynomial of $\rho(z)(Frob_l)$ for any $l \notin S$ is equal to $X^2 - T_l(z) X + l S_l(z)$.
In particular $l S_l(z)=det(\rho(z)(Frob_l))$. Let $S_z$ be the one-dimensional pseudo-representation $ det(\rho(z)(\cdot))$. 
\begin{prop} There exists a unique continuous pseudo-character 
$$\mathcal{S}: G_{\Q,S}\rightarrow \mathcal{H}_{\widetilde{\D}}^*$$
of dimension 1 such that for each $z \in Z$ the evaluation of $\mathcal{S}$ at $z$ is equal to $S_z$. 
\end{prop}
\begin{proof} Apply  Proposition 7.1.1 of \cite{Chenevier} using the functions $(l S_l)_{l \notin S}$.
\end{proof}
\begin{rem} For $g \in G_{\Q,S}$ the rigid analytic functions $\mathcal{S}(g)$ and 
\[ g\mapsto (\mathcal{T}^2(g)-\mathcal{T}(g^2))/2
\]
are the same as they agree on $Z$. Therefore we have that for any $x \in \widetilde{\D}(\Qbar_p)$, $S_x=det(\rho(x))$.
\end{rem}

For later purposes define 
\begin{equation}
\begin{aligned}
\mathcal{P}:G_{\Q,S} &\rightarrow &\mathcal{O}(\widetilde{\D}) \\
g &\mapsto& \frac{\mathcal{T}(g)^2}{\mathcal{S}(g)}.
\label{psr}
\end{aligned}
\end{equation}

\subsection{Eigenvarieties of idempotent type for $G$}
We may also specialize the general theory of \cite{david} to produce eigenvarieties for $G$, the group of norm one elements in $\widetilde{G}$: 

Let $e:= \prod{e_l} \in C_c^\infty (G(\A_f^{p}),E)$ be an idempotent. 
For $Y\subset \mathcal{W}$ an affinoid subspace and $k \geq k(Y)$ we have the Banach $\mathcal{O}(Y)$-modules $\mathcal{C}(Y,k):= \mathcal{C}(Y,1,k)$ from Def. 2.3.1 of \cite{david}. 
Just as before we can build the space $M(e,Y,k)$ of $\mathcal{C}(Y,k)$-valued automorphic forms of $G$.
Any $u \in \mathcal{H}_S$ that is supported in $G(\A_f^{S})\times I\Sigma^{++}I$ gives rise to a compact operator on $M(e,Y,k)$. Again we have links 
$\alpha_{YZ}: M(e,Z) \rightarrow M(e,Y) \widehat{\otimes}_{\mathcal{O}(Y)} \mathcal{O}(Z)$ for $Z \subset Y \subset \mathcal{W}$ two open affinoid subdomains. 
We get the following
\begin{prop} For any idempotent $e \in C_c^\infty (G(\A_f^{p}),E)$ and $u \in \mathcal{H}_S$ as above, the datum $(\mathcal{W}, M(e)_\bullet, \alpha_\bullet, \mathcal{H}_S,u)$ is an  eigenvariety datum. We denote by $\mathcal{D}_u(e)$ the associated eigenvariety.
\end{prop}

Let $k \in \Z_{\geq 0}$. Analogously to the definitions in Section \ref{subsec: firsteigenvarieties} we say:
\begin{defn}
\begin{enumerate}
\item Let $\pi_p$ be an irreducible smooth representation of $\GL_2(\Q_p)$ such that $\pi_p^I \neq 0$. A character $\chi: \Sigma \rightarrow \C^*$ is called an accessible refinement of $\pi_p$ if $\chi \delta_B^{-1/2}:\Sigma \rightarrow \C^*$ occurs in $\pi_p^I$.
\item A $p$-refined automorphic representation of weight $k$ of $G(\A)$ is a pair $(\pi, \chi)$ such that
\begin{itemize}
	\item $\pi$ is an automorphic representation of $G(\A)$, 
	\item $\pi_p$ has a non-zero $I$-fixed vector and $\chi: \Sigma \rightarrow \C^*$ is an accessible refinement of $\pi_p$, 
	\item $\pi_\infty \cong  \operatorname{Sym}^k(\C^2)$. 
\end{itemize}
\end{enumerate}
\end{defn}
Let $e:= \prod{e_l} \in C_c^\infty (G(\A_f^{p}),\Qbar)$ be an idempotent and assume $\iota_p(e)$ takes values in~$E$. For a $p$-refined automorphic representation $(\pi,\chi)$ of $G(\A)$ such that $\iota_\infty(e)\cdot \pi_f^p \neq 0$ define the character 
$$ \psi_{(\pi,\chi)}:\mathcal{H}_S\rightarrow \Qbar_p, \psi_{(\pi,\chi)}:= \chi \delta_B^{-1/2} \delta_k \otimes \chi_{ur},$$
where $\delta_k: \mathcal{A}_p(G) \rightarrow \Qbar_p$ is given by $(p^a,p^{-a}) \mapsto p^{-ka}$, and $\chi_{ur}$ is the obvious character. Again one shows that a $p$-refined automorphic representation gives rise to a point on $  \D_u(e):=\D_u(\iota_p(e))$.

\subsection{A first auxiliary eigenvariety}
\label{subsec: aux}
As before let $\widetilde{e}=\prod \widetilde{e}_l \in C_c^\infty (\widetilde{G}(\A_f^{p}),\Qbar)$ be an idempotent and assume $\iota_p(\widetilde{e})$ takes values in $E$. 
Choose a finite set of primes $S$ containing all the bad primes.
Fix $u \in \widetilde{\mathcal{H}}_S$ of the form $ \mathbf{1}_{[\widetilde{I}z\widetilde{I}]} \otimes 1_{\widetilde{\mathcal{H}}_{ur,S}} $ for some $z \in \Sigma^{++}$, so in particular $u$ is in the image of the morphism $\lambda:\mathcal{H}_S\hookrightarrow \widetilde{\mathcal{H}}_S$ constructed in Section \ref{subsec: iwahori}. 
Over an affinoid $X$ in $\widetilde{\mathcal{W}}$ we have the modules $M(\widetilde{e},X)$ as before.
By composing the morphism $\widetilde{\mathcal{H}}_S \rightarrow \operatorname{End}_{\mathcal{O}(X)}(M(\widetilde{e},X))$ with the inclusion $\lambda$  we get a morphism $\mathcal{H}_S \rightarrow \operatorname{End}_{\mathcal{O}(X)}(M(\widetilde{e},X))$. Note that $u$ acts as a compact endomorphism. We abbreviate $\widetilde{\mathcal{D}}:= \widetilde{\mathcal{D}}_u(\widetilde{e}).$
\begin{prop}
\begin{enumerate}[(a)]
	\item For any open affinoid $X \subset \widetilde{\mathcal{W}}$ there exists an eigenvariety $\D'_X$ attached to $(X,M(\widetilde{e},X),\mathcal{H}_S,u)$.  
Moreover there exists a reduced separated rigid analytic space $\mathcal{D}'$ over $\widetilde{\mathcal{W}}$ such that for any open affinoid $X \subset \widetilde{\mathcal{W}}$ the pullback of $\D'$ to $X$ is canonically isomorphic to $\D'_X$. 
\item There exists a unique morphism $\zeta':\widetilde{\mathcal{D}} \rightarrow \D'$ over $\widetilde{\mathcal{W}}$ such that the diagram
$$\xymatrix{
\mathcal{H}_S \  \ar[d] \ar@{^{(}->}[r]^\lambda & \widetilde{\mathcal{H}}_S \ar[d] \\
  \mathcal{O}(\D') \ar[r]^{\zeta'^*}  &\mathcal{O}(\widetilde{\D}) }$$
commutes. In particular for any finite extension $E'/E$ the map $\widetilde{\mathcal{D}}(E')\rightarrow \D'(E')$ is given by restricting a system of Hecke eigenvalues $\widetilde{\mathcal{H}}_S\rightarrow E'$ to $\mathcal{H}_S$.
\end{enumerate}
\label{D'}
\end{prop}
\begin{proof} The first claim or part (a) follows by applying Proposition \ref{EV} to the data $(X,M(\widetilde{e},X),\mathcal{H}_S,u)$. As the links $\alpha_{XY}$ are obviously also links for the action of $\mathcal{H}_S$ we can glue the local pieces, i.e.\ we can apply Theorem \ref{gEV}, which proves part (a). 

For the proof of (b) we work with the admissible cover $\widetilde{\D}(Y)_{Y\in \mathcal{C}_X}$ of the eigenvariety $\widetilde{\mathcal{D}}_X$ for $X \subset \widetilde{\mathcal{W}}$ affinoid open that we described in Remark \ref{admcover} and use the notation of that Remark. We may assume $Y$ is such that $U=f(Y)$ is connected. By construction $\widetilde{\D}(Y)$ is an affinoid $\operatorname{Sp}(\widetilde{\mathcal{H}}_S(Y))$. Here $\widetilde{\mathcal{H}}_S(Y)$ is the reduced image of $\widetilde{\mathcal{H}}_S \otimes_E \mathcal{O}(U)$ in the endomorphism ring $\operatorname{End}_{\mathcal{O}(U)}((M(\widetilde{e},X)\hat{\otimes}_{\mathcal{O}(X)}\mathcal{O}(U))_{fs})$.  
Analogously let $\mathcal{D}'(Y)$ denote the part of $\mathcal{D}'_X$ above $Y$. By construction $\mathcal{D}'(Y)= \operatorname{Sp}(\mathcal{H}_S(Y))$, where $\mathcal{H}_S(Y)$ is the reduced image of $\mathcal{H}_S\otimes_E \mathcal{O}(U)$ in the endomorphism ring $\operatorname{End}_{\mathcal{O}(U)}((M(\widetilde{e},X)\hat{\otimes}_{\mathcal{O}(X)}\mathcal{O}(U))_{fs})$.

Let $X$ be as above and $Y \in \mathcal{C}_X$. Then $\lambda: \mathcal{H}_S \hookrightarrow \widetilde{\mathcal{H}}_S$ induces an injection $\lambda^*_Y:\mathcal{H}_S(Y) \hookrightarrow \widetilde{\mathcal{H}}_S(Y)$. Let $\zeta'_{Y}:\widetilde{\mathcal{D}}(Y) \rightarrow \mathcal{D}'$ be the morphism given by composing~$\lambda_Y$ with the natural inclusion $\mathcal{D}'(Y)\subset \mathcal{D}'$. The affinoid spaces $\widetilde{\mathcal{D}}(Y), Y\in \mathcal{C}_X$, as $X$ runs through an admissible cover of $\widetilde{\mathcal{W}}$, form an admissible cover of $\widetilde{\mathcal{D}}$. 

It is easy to see that the $\zeta'_{Y}$ glue. Namely choose a nested cover of $(X_i)_{i \in \mathbb{N}}$ of $\widetilde{\mathcal{W}}$ and let $\widetilde{\mathcal{D}}(Y_1), Y_1 \in \mathcal{C}_{X_i}$ and $\widetilde{\mathcal{D}}(Y_2), Y_2 \in \mathcal{C}_{X_j}$ be two affinoid spaces of the cover of~$\widetilde{\D}$ and assume $i\leq j$. Then $Y_1 \in \mathcal{C}_{X_j}$ and by Lemma 5.2 of \cite{kevin} we know that $\widetilde{\mathcal{D}}(Y_1)\cap \widetilde{\mathcal{D}}(Y_2) = \widetilde{\mathcal{D}}(Y_1 \cap Y_2)$. Furthermore, using $\widetilde{\mathcal{H}}_S(Y_1\cap Y_2) \cong \widetilde{\mathcal{H}}_S(Y_1)\widehat{\otimes}_{\mathcal{O}(Y_1)} \mathcal{O}(Y_1 \cap Y_2)$, we see that ${\zeta'_{Y_1}}|_{\widetilde{\mathcal{D}}(Y_1 \cap Y_2)}=\zeta'_{{Y_1}\cap{Y_2}} $. 
It is clear that $\zeta'$ makes the diagram in $(b)$ commutative. 

Finally the uniqueness of $\zeta'$ follows from Lemma \ref{7.2.7} and the commutativity of the diagram as $\widetilde{\D}$ and $\D'$ are reduced.
\end{proof}
Recall how we defined a set of classical points $\mathcal{Z}_{\widetilde{e}} \subset \Hom_{ring}(\widetilde{\mathcal{H}}_S,\Qbar_p) \times \widetilde{\mathcal{W}}(\Qbar_p)$ in Equation (\ref{cp}). Let $\mathcal{Z}_{\widetilde{e}}' \subset \Hom_{ring}(\mathcal{H}_S,\Qbar_p) \times \widetilde{\mathcal{W}}(\Qbar_p) $ be the set of all pairs $(\psi_{(\pi,\chi)}\circ \lambda,\underline{k})$ such that $(\psi_{(\pi,\chi)},\underline{k}) \in \mathcal{Z}_{\widetilde{e}}$.

\begin{prop}
$\mathcal{D}'$ is the eigenvariety for $\mathcal{Z}'_{\widetilde{e}}$ in the sense of Definition \ref{evforz}.
\label{Z'e}
\end{prop}

\begin{proof} As before let $Z$ be image of the map $\mathcal{Z}_{\widetilde{e}} \hookrightarrow \widetilde{\mathcal{D}}(\Qbar_p)$. We prove that $Z':= \zeta'(Z)$ is accumulation and Zariski-dense in $\mathcal{D}'$. We use the same notation as in the proof of Proposition \ref{D'}. For the accumulation property let $z \in Z'$ be a point and let $V$ be an open affinoid neighbourhood of $z$. Without loss of generality (see the proof of Theorem 3.1.1 of \cite{taibi}) we may assume that $V$ is a connected component of an affinoid of the form $\D'(Y)$ for $Y\in\mathcal{C}_X$ for some $X \subset \widetilde{\mathcal{W}}$ affinoid open. 
Now $Z$ is Zariski-dense in $\widetilde{\mathcal{D}}(Y)$ and the morphism $\zeta:\mathcal{O}(\mathcal{D}'(Y)) \hookrightarrow \mathcal{O}(\widetilde{\D}(Y))$ is injective.
Therefore $Im(Z)$ is Zariski-dense in $\mathcal{D}'(Y)$ and therefore also Zariski-dense in the connected component $V$. This shows the accumulation property.

The Zariski-density also follows from this as the $\mathcal{D}'(Y)$ admissibly cover $\mathcal{D}'$.
\end{proof}

\subsubsection{Some properties of the morphism $\zeta'$} \label{subsubsec: twists}
We discuss when two points $x, y \in \widetilde{\mathcal{D}}(\Qbar_p)$ have the same image under $\zeta'$. 

If $x$ and $y$ are in $Z$ then there exist $p$-refined automorphic representations of $\widetilde{G}(\A)$, say $(\pi,\chi)$ and $(\pi',\chi')$ unramified at all places not in $S$ that give rise to the points~$x$ and $y$. Then $\zeta'(x)=\zeta'(y)$ implies that the weights agree, so 
 $$ \omega(x)=\omega(y)$$
and therefore $\pi_\infty \cong \pi'_\infty$. Furthermore we have 
$$\iota_\infty\circ\iota_p^{-1}({\psi_x}|_{\mathcal{H}_{ur}})=\iota_\infty\circ\iota_p^{-1}({\psi_y}|_{\mathcal{H}_{ur}}).$$
It is not hard to show that this implies that for all $l \notin S$, $\pi_l$ is an unramified twist of~$\pi'_l$. 
Then a result of Ramakrishnan (Theorem 4.1.2 of \cite{Ramakrishnan}) applied to the Jacquet--Langlands transfer of $\pi$ and $\pi'$ shows that $\pi$ and $\pi'$ are in fact global twists of each other. We omit the details as we prefer to give a more elementary characterization of the points that are getting identified under $\zeta'$ in terms of the associated Galois representation.
 
For that let $x$ and $y \in \widetilde{\D}(\Qbar_p)$ be arbitrary points such that $\zeta'(x)=\zeta'(y)$. Recall the definition of $\mathcal{P}:= \mathcal{T}^2/\mathcal{S}$ from Section \ref{sec: 3.2.1}, Equation (\ref{psr}). For a point $z \in \widetilde{\D}(\Qbar_p)$ we let $P_z$ be the specialization $T_z^2/S_z$. As the restrictions of $\psi_x$ and $\psi_y$ to $\mathcal{H}_{ur}$ agree we have (use Equation (\ref{urops}))
$$P_x(Frob_l)=\frac{tr^2(\rho(x)(Frob_l))}{det(\rho(x)(Frob_l))}=\frac{tr^2(\rho(y)(Frob_l))}{det(\rho(y)(Frob_l))}=P_y(Frob_l)$$
for all $l \notin S$ and so by continuity we have that $P_x=P_y$. 
Let $\eta:\GL_2(\Qbar_p)\rightarrow \PGL_2(\Qbar_p)$ denote the natural map. Recall from Section 7.2 of \cite{Chenevier} that there exists an open subspace $\widetilde{\D}^{irr}\subset \widetilde{\D}$ such that $x \in \widetilde{\D}(\Qbar_p)$ lies in $\widetilde{\D}^{irr}$ if and only if $\rho(x)$ is irreducible. 
\begin{prop}
Assume $x$ and $y \in \widetilde{\D}^{irr}(\overline{\Q}_p)$ are such that $\zeta'(x)= \zeta'(y)$. Then 
$$ \eta \circ \rho(x) \cong \eta \circ \rho(y).$$
\end{prop}
\begin{proof}
The equation $P_x=P_y$ implies that $\eta(\rho(x)(g))$ is conjugate to $\eta(\rho(y)(g))$ for all $g \in G_\Q$.
We have an injection $i:\PGL_2(\Qbar_p) \hookrightarrow \GL_3(\Qbar_p)$ with image $\operatorname{SO}_3(\Qbar_p)$.
Composing $\rho(x)$ and $\rho(y)$ with this injection we see that $i(\rho(x)(g))$ is conjugate to $i(\rho(y)(g))$ for all $g$. 
Therefore $i\circ\eta \circ \rho(x)$ and $i\circ \eta \circ \rho(y)$ are conjugate by an element~$B$ in $\GL_3(\Qbar_p)$. But then in fact $\eta \circ \rho(x)$ and $ \eta \circ \rho(y)$ are conjugate in $\PGL_2(\Qbar_p)$, see e.g.\ Theorem B and p.34--35 of \cite{appram}.
\end{proof}

\subsection{A second auxiliary eigenvariety}
Let $e \in C^{\infty}_c(G(\A_f^p),E)$ be an idempotent, $S$ a fixed set of bad primes, $\mathcal{H}_S$ the Hecke algebra associated to $S$ and let $\mathcal{D}:=\mathcal{D}_u(e)$ be an eigenvariety of idempotent type $e$ for $G$.
Define $\D'':= \widetilde{\mathcal{W}}\times_{\mathcal{W}} \D$ to be the pullback of $\D$ under the morphism $\mu:\widetilde{\mathcal{W}}\rightarrow \mathcal{W}$ from Section \ref{sec: weights}.
The space $\mathcal{W}$ is nested. Choose a nested cover, e.g.\ for $i\in \mathbb{N}$, let $X_i\subset \mathcal{W}$ be the affinoid subspace that is the disjoint union of $p-1$ (resp.\ 2) closed disks of radius $p^{-1/i}$ if $p \neq 2$ (resp.\ if $p =2$). Then $\mathcal{W}=\bigcup_{i \in \mathbb{N}} X_i$ and for $i<j$ the map $\mathcal{O}(X_j) \rightarrow \mathcal{O}(X_i)$ is compact. Recall from Lemma \ref{nested} that for any $V\subset \widetilde{\mathcal{W}}$ affinoid open, there exists $X \subset \mathcal{W}$ affinoid open such that $\mu(V) \subset X$. 
\begin{lem}
Let $V\subset \widetilde{\mathcal{W}}$ be an open affinoid and $X \subset \mathcal{W}$ be an affinoid open such that $\mu(V) \subset X$. Then $\D''_V:= V \times_{\widetilde{\mathcal{W}}} \D''$ is the eigenvariety for the data $(V,M(e,X)\widehat{\otimes}_{\mathcal{O}(X)}\mathcal{O}(V), \mathcal{H}_S,u)$.
\end{lem}
\begin{proof}
Note that $\D''_V \cong V \times_X \D_X$ and that the induced map $\mu_V: V\rightarrow X$ is flat. Therefore the result follows from Lemma 5.4 and 5.5 of \cite{kevin}. 
\end{proof}
\begin{lem} Let $R' \rightarrow R$ be a morphism of affinoid $E$-algebras. Assume $M $ and $N$ are Banach $R'$-modules with a continuous action of a commutative $E$-algebra \textbf{T} and assume a fixed element $u \in \textbf{T}$ acts compactly. Assume $\alpha: M \rightarrow N$ is a primitive link. Then the base change $\alpha_R: M\widehat{\otimes}_{R'}R \rightarrow N \widehat{\otimes}_{R'} R $ is also a primitive link.
\end{lem}
\begin{proof} This follows immediately from the definitions.
\end{proof}
In fact $\mathcal{D}''$ is a global eigenvariety. 
For any open affinoid $V \subset \widetilde{\mathcal{W}}$  choose the minimal $i \in \mathbb{N}$ such that $\mu(V) \subset X_i$, denote it by $i_V$ and define $X_V:=X_{i_V}$. 
We let $M(V)$ be the orthonormalizable Banach $\mathcal{O}(V)$-module 
$$M(V):= M(e,X_V)\widehat{\otimes}_{\mathcal{O}(X_V)}\mathcal{O}(V).$$
It satisfies property (Pr) (see \cite{kevin} Lemma 2.13). 
If $U \subset V \subset \widetilde{\mathcal{W}}$ is an inclusion of open affinoids, then $i_{U}\leq i_{V}$. We have a link 
$$\alpha_{X_{U} X_{V}}: M(e,X_{U})\rightarrow M(e,X_V)\widehat{\otimes}_{\mathcal{O}(X_V)}\mathcal{O}(X_U).$$
The previous lemma implies that the pullback under $U\rightarrow X_U$
$$\alpha_{UV}:= (\alpha_{X_{U} X_{V}})_{\mathcal{O}(U)}: M(U) \rightarrow M(V)\hat{\otimes}_{\mathcal{O}(V)}\mathcal{O}(U) $$
is a link.
We abbreviate this data as $(\widetilde{\mathcal{W}},(\mu^*M)_\bullet, \alpha_\bullet, \mathcal{H}_S,u)$. The fact that eigenvarieties are unique then implies the following proposition.
\begin{prop} The rigid space
$\mathcal{D}''$ is the global eigenvariety associated to the datum $(\widetilde{\mathcal{W}},(\mu^*M)_\bullet, \alpha_\bullet, \mathcal{H}_S,u)$.
\end{prop}

\section{Classical functoriality - Background from Labesse Langlands}
In this section we explain the results of \cite{LL} that we need below.
Let $F$ be a totally real number field, $B$ a definite quaternion algebra over $F$, $\widetilde{G}$ the algebraic group over $F$ defined by the units in $B$ and $G$ the algebraic group defined by the elements of reduced norm 1. As before $S_B$ denotes the set of places, where $B$ is ramified.

\subsection{Local and global $L$-packets}
Let $\widetilde{H}$ be a group and $H$ a normal subgroup. For a representation $\pi$ of $H$ and $ g \in \widetilde{H}$ we denote by $^g\pi$ the representation given by $^g\pi(h)= \pi(g^{-1}hg)$ for $h \in H$.
\begin{defn}
Let $v$ be a finite place of $F$. 
A local $L$-packet of $G(F_v)$ is a finite set $\{\pi_i: i \in I\}$ of irreducible admissible representations of $G(F_v)$, such that there exists an irreducible admissible representation $\widetilde{\pi}$ of $\widetilde{G}(F_v)$ and a positive integer $c$ with 
\begin{equation} 
\widetilde{\pi}|_{G(F_v)} \cong \bigoplus_{i \in I} c \pi_i .
\label{locpac}
\end{equation}
Starting from an irreducible admissible representation $\widetilde{\pi}$ of $\widetilde{G}(F_v)$ we denote by $\Pi(\widetilde{\pi})$ the $L$-packet $\{\pi_i: i \in I\}$ defined by $\widetilde{\pi}$.
\end{defn}

\begin{rem}
We will make use of the following facts:
\begin{itemize}
\item Any irreducible smooth representation of $G(F_v)$ occurs in the restriction of a representation from $\widetilde{G}(F_v)$ and therefore belongs to an $L$-packet, see Lemma 2.5 of \cite{LL}. In particular $L$-packets partition the set of equivalence classes of admissible irreducible representations into finite sets. 
	\item Size of local packets: For $v \notin S_B$ $L$-packets of $G(F_v)$ are of size $1,2$ or $4$, see p.739 of \cite{LL}, and the multiplicity $c$ in (\ref{locpac}) is equal to 1 (see Lemma 2.6 of \cite{LL}). For $v \in S_B$ $L$-packets of $G(F_v)$ are of size $1$ or $2$ and $c \in \{1,2\}$ (see p.751 and Lemma 7.1 of \cite{LL}).
	
	\item If $\widetilde{\pi}$ is an unramified representation of $\GL_2(F_v)$, there exists a unique member in $\Pi(\widetilde{\pi})$ which has a non-zero $\SL_2(\mathcal{O}_{F_v})$-fixed vector. This is obvious if $\widetilde{\pi}$ is a 1-dimensional representation. Otherwise 
	\[\widetilde{\pi}\cong Ind_{\widetilde{B}}^{\GL_2(F_v)}(\chi_1,\chi_2)\cong Ind_{\widetilde{B}}^{\GL_2(F_v)}(\chi_1\chi_2^{-1},1)\otimes \chi_2,\] 
	where $\chi_1$ and $\chi_2$ are unramified. The restriction of $\widetilde{\pi}$ to $\SL_2(F_v)$ is isomorphic to $Ind_{B}^{\SL_2(F_v)}(\chi_1\chi_2^{-1})$. 
	By Proposition 3.2.4 of \cite{lansky} we have that $\dim(\widetilde{\pi}^{\SL_2(\mathcal{O}_{F_v})})=1$. We denote the unique member of the $L$-packet of $\widetilde{\pi}$ that has a non-zero $\SL_2(\mathcal{O}_{F_v})$-fixed vector by $\pi^0$ and in a global context by $\pi^0_v$. 
\item Any set of representatives $\{g_i: i \in I \}$ of the cosets $\widetilde{G}(F_v)/G(F_v)F_v^* $ permutes the irreducible subrepresentation of $\widetilde{\pi}|_{G(F_v)}$.
\end{itemize}
\end{rem}

\begin{rem}
At the archimedean places one can define $L$-packets in the same way. As we are only considering definite quaternion algebras the situation simplifies. The irreducible representations of $\widetilde{G}(\R)$ are finite-dimensional, and they stay irreducible when restricted to the subgroup $G(\R)$ of norm one elements, as $\widetilde{G}(\R)$ is generated by $G(\R)$ and its centre. The $L$-packets are therefore singletons.
\end{rem}

\begin{defn}
Let $\widetilde{\pi}=\otimes \widetilde{\pi}_v$ be a discrete automorphic representation of $\widetilde{G}(\A)$. The global $L$-packet of $G(\A)$ associated to $\widetilde{\pi}$ is 
$$\Pi(\widetilde{\pi}):= \{\otimes \pi_v \ |\  \pi_v \in \Pi(\widetilde{\pi}_v), \pi_v=\pi^0_v \text{ for almost all } v \}.$$
\end{defn}

\begin{rem} 
Let $L/F$ be a quadratic extension and $\widetilde{\theta}: L^*\backslash \A^*_L \rightarrow \C^*$ be a Gr{\"o}{\ss}en-character, which does not factor through the norm $N^L_F(\cdot)$. In \cite{JL} Jacquet and Langlands show how to associate a cuspidal automorphic representation $\tau(\widetilde{\theta})$ of $\GL_2(\A_F)$ to $\widetilde{\theta}$. 

Via class field theory the character $\widetilde{\theta}$ gives rise to a representation of the global Weil group $W_{L}$ of $L$ and to a two-dimensional representation $Ind(W_F,W_L,\widetilde{\theta})$ of the Weil group $W_F$ of $F$. The representation $\tau(\widetilde{\theta})$ is characterized uniquely by an equality of $L$-functions and $\epsilon$-factor of (twists of) $\tau(\widetilde{\theta})$ and (twists of) $Ind(W_F,W_L,	\widetilde{\theta})$ (we refer to \S 12 of \cite{JL} for details regarding the construction and characterization).
\end{rem}

If $\tau(\widetilde{\theta})$ is in the image of the global Jacquet--Langlands transfer from $\widetilde{G}$, i.e.\ if $\tau(\widetilde{\theta})_v$ is a discrete series representation for all $v \in S_B$, we denote the preimage $JL^{-1}(\tau(\widetilde{\theta}))$ by $\pi(\widetilde{\theta})$. 
The global $L$-packet defined by $\pi(\widetilde{\theta})$ depends only on the restriction of $\widetilde{\theta}$ to the subgroup of $\A^*_L$ of norm 1 elements. We denote this restriction by $\theta$ and the associated $L$-packet by $\Pi(\theta)$.

\begin{rem} 
Write $L=F(\sqrt{d})$ for some $d \in F^* \backslash (F^*)^2$. The embedding 
$$ L \hookrightarrow M_2(F) , (a+ \sqrt{d}b) \mapsto  \begin{pmatrix} a&bd\\ b&a \end{pmatrix}$$
identifies $L^*$ with the $F$-points of a maximal torus $\widetilde{T}$ of $\GL_2/F$ and the subgroup $L^1$ of norm one elements with $T(F):= \widetilde{T}(F) \cap \SL_2(F)$, which are the $F$-points of an elliptic endoscopic group $T$ of $G$. With the exception of $\SL_2$, all elliptic endoscopic groups are of this form.
The above shows that under certain conditions a character~$\theta$ of $T(\A)$, which is trivial on $T(F)$, gives rise to a global $L$-packet $\Pi(\theta)$ of $G(\A)$.
\label{endgps}
\end{rem}

\begin{defn}
We say that an admissible irreducible representation $\pi$ of $G(\A)$ is endoscopic if there exist $T$ and $\theta$ as above such that $\pi \in \Pi(\theta)$. We refer to the $L$-packet $\Pi(\theta)$ as an endoscopic packet and call $L$-packets, which are not of this form, stable $L$-packets.
\end{defn}

\subsection{Multiplicity formulas}

In \cite{LL} Labesse and Langlands determine which elements of a global $L$-packet $\Pi(\widetilde{\pi})$ are automorphic, i.e.\ they derive formulas for the multiplicities $m(\pi)$ for $\pi \in \Pi(\widetilde{\pi})$ in the discrete automorphic spectrum. We collect the results that we need below.

For a maximal non-split torus $\widetilde{T}/F$ in $\GL_2$ and a character $\widetilde{\theta}$ of $\widetilde{T}(\A)$ we let 
$\overline{\widetilde{\theta}}$ be the conjugate character of $\widetilde{T}(\A)$, i.e.\ the character defined by
$$\overline{\widetilde{\theta}}(\gamma) = \widetilde{\theta}(w\gamma w^{-1}) = \widetilde{\theta}(\overline{\gamma}), $$
for all $\gamma \in \widetilde{T}(\A)$. Here $w$ is any element in the normalizer of $\widetilde{T}(F)$ which is not in $\widetilde{T}(F)$, e.g.\ using the notation of Remark \ref{endgps} we may take $w = \left(_1 \ ^{-d} \right)$.
For a character $\theta$ of $T(\A)$ we define $\overline{\theta}$ in the same way.
\begin{defn}[cf.\ \cite{LL} p.762--764]
Let $\widetilde{T}$ and $T$ be as above and $\theta$ a character of $T(F)\backslash T(\A)$. We say that 
\begin{itemize}
	\item $\theta$ is of type $(a)$ if $\theta \neq \overline{\theta}$, 
	\item $\theta$ is of type $(b)$ if $\theta = \overline{\theta}$, but $\theta $ does not extend to a character $\widetilde{\theta}$ of $\widetilde{T}(F) \backslash \widetilde{T}(\A) $ satisfying  $\widetilde{\theta}(\gamma) = \widetilde{\theta}(\overline{\gamma})$ for all $\gamma \in \widetilde{T}(\A)$ and
	\item $\theta$ is of type $(c)$ if $\theta$ can be extended to a character $\widetilde{\theta}$ of $\widetilde{T}(F) \backslash \widetilde{T}(\A)$ satisfying $\widetilde{\theta}(\gamma) = \widetilde{\theta}(\overline{\gamma})$ for all $\gamma \in \widetilde{T}(\A)$.
\end{itemize}
\end{defn}

\begin{lem}
Assume $\Pi(\theta)$ is the $L$-packet associated to an automorphic representation $\pi(\widetilde{\theta})$ of $\widetilde{G}(\A)$. Then $\theta$ is necessarily of type $(a)$.
\label{typea}
\end{lem}
\begin{proof}
The representation of $\GL_2(\A)$ associated to a character $(c)$ is a global principal series representation, in particular not a discrete automorphic representation. 

Assume that $\theta$ is of type $(b)$. By assumption $S_B$ contains all archimedean places. Let $v$ be an infinite place.
Let $\widetilde{\theta}$ be an extension of $\theta$ to $\widetilde{T}(\A)$. Let $\psi$ be the character $\gamma \mapsto \widetilde{\theta}(\gamma / \overline{\gamma}) = \theta(\gamma/\overline{\gamma})$ of $\widetilde{T}(\A)$. We view the local component $\widetilde{\theta}_v$ as a character of $\C^*$, the local Weil group of $\C$. 
We can write $ \widetilde{\theta}_v = \psi_v \cdot \chi$ where $\chi$ satisfies $\chi(x) = \chi(\overline{x})$ for all $x \in \C^*$. 
But $\psi$ is a quadratic character, so the local component $\psi_v$ is trivial. In particular $\widetilde{\theta}_v$ factors through the norm and so the associated representation $\tau(\widetilde{\theta}_v)$ of $\GL_2(\R)$ cannot be a discrete series representation. But we assumed that $\widetilde{\theta}$ gives rise to an automorphic representation $\pi(\widetilde{\theta})$ on $\widetilde{G}(\A)$ which implies that $\tau(\widetilde{\theta}_v) \cong JL(\pi(\widetilde{\theta}_v))$ is a discrete series representation -- a contradiction. Therefore $\theta$ is of type~$(a)$. 
\end{proof}

For an irreducible admissible representation $\pi$ of $G(\A)$, let $m(\pi)$ be the multiplicity with which $\pi$ shows up in the discrete automorphic spectrum. For any element~$\pi$ of an endoscopic packet Labesse and Langlands define integers $\left\langle 1, \pi \right\rangle $ and $\left\langle \epsilon, \pi \right\rangle$.  
We refer to Section $6$ and $7$ of \cite{LL} for details regarding the definition. We remark that they are defined as products $ \left\langle 1, \pi \right\rangle := \prod_v{\left\langle 1, \pi_v \right\rangle } $ and $ \left\langle \epsilon, \pi \right\rangle := \prod_v{\left\langle \epsilon, \pi_v \right\rangle},$
where $\left\langle 1, \pi_v \right\rangle \in \{1,2\}$ and $\left\langle \epsilon, \pi_v \right\rangle \in \{-1,0,1\}$ and both are equal to 1 for almost all places.

\begin{thm}[\cite{LL}]
Let $\pi$ be an irreducible admissible representation of $G (\A)$ and assume $\pi \in \Pi(\widetilde{\pi})$ for some automorphic representation $\widetilde{\pi}$ of $\widetilde{G}(\A)$. 

\begin{enumerate}
	\item If $\pi$ is endoscopic, then 
	$$ m(\pi) = \frac{1}{2}( \left\langle 1,\pi \right\rangle + \left\langle \epsilon, \pi \right\rangle).$$
	\item If $\pi$ is not endoscopic, then $m(\pi) = m(^g\pi) \neq 0$ for all $g \in \widetilde{G}(\A)$. 	
\end{enumerate}
\label{multiplicities}
\end{thm}
\begin{proof}
By Lemma \ref{typea} there are no endoscopic $L$-packets of type $(b)$. 
The above results are then contained in Proposition 7.2 and 7.3 of \cite{LL}  (the number $d(\pi)$ occurring there is equal to $1$ in our case).
\end{proof}

\begin{prop}
Assume the quaternion algebra $B$ ramifies at a finite place.
Fix a finite place $w \in S_B$. Let $\Pi(\widetilde{\pi})$ be a global $L$-packet and let $\pi= \otimes \pi_v \in \Pi(\widetilde{\pi})$ be an element. Then there exists $\tau \in \Pi(\widetilde{\pi}_{w})$ such that 
$$ m\left(\tau \otimes \bigotimes_{v \neq w} \pi_v \right) > 0.$$
\label{switch}
\end{prop}

\begin{proof}
If $m(\pi)>0$ there is nothing to show. We may therefore assume $m(\pi)=0$. But then $\pi$ is endoscopic, so $\Pi(\widetilde{\pi})=\Pi(\theta)$ as above and in particular $\left\langle \epsilon, \pi_w \right\rangle \neq 0$. But then Lemma 7.1 of \cite{LL} and the discussion on p.783 of loc.cit.\ implies that the $L$-packet $\Pi(\widetilde{\pi}_w) = \{\pi_w,\pi'_w\}$ is of size two, 
$$\left\langle \epsilon, \pi_w \right\rangle , \left\langle \epsilon, \pi'_w \right\rangle \in \{-1,1\}$$
and the numbers are distinct. The result follows from the formula in Theorem \ref{multiplicities}.
\end{proof}

\subsection{Compatible tame levels}
Let $v$ be a finite prime of $F$.
\begin{defn} 
\begin{enumerate}
\item We say that two idempotents $\widetilde{e} \in C^\infty_c(\widetilde{G}(F_v), \Qbar)$ and $e \in C^\infty_c(G(F_v),\Qbar)$ are locally Langlands compatible if they satisfy: For any irreducible smooth representation $\widetilde{\pi}$ of $\widetilde{G}(F_v)$, such that $\iota_\infty(\widetilde{e}) \cdot \widetilde{\pi} \neq 0$, there exists an element $\pi \in \Pi(\widetilde{\pi})$, such that $\iota_\infty(e) \cdot \pi \neq 0$.
	\item We say that $\widetilde{e}$ is strongly locally Langlands compatible with $e$ if: For any smooth representation $\widetilde{\pi}$ of $\widetilde{G}(F_v)$, such that $\iota_\infty(\widetilde{e}) \cdot \widetilde{\pi} \neq 0$, every element $\pi$ of $\Pi(\widetilde{\pi})$ has the property that $\iota_\infty(e) \cdot \pi \neq 0$.
	\item We call two compact open subgroups $\widetilde{K} \subset \widetilde{G}(F_v)$ and $K \subset G(F_v)$ (strongly) locally Langlands compatible if the idempotents $e_{\widetilde{K}}$ and $e_K$ are. 
\end{enumerate}
\end{defn}

\begin{defn} \label{GLCdef}
Two idempotents $\widetilde{e} \in C^\infty_c(\widetilde{G}(\A^p_f), \Qbar)$ and $e \in C^\infty_c(G(\A^p_f),\Qbar)$ are called globally Langlands compatible if the following holds: For any discrete automorphic representation $\widetilde{\pi}$ of $\widetilde{G}(\A)$, such that $\iota_\infty(\widetilde{e}) \cdot \widetilde{\pi}^p_f \neq 0$ and any $\tau \in \Pi(\widetilde{\pi}_p)$, there exists an element $\pi$ in the packet $\Pi(\widetilde{\pi})$, such that 
\begin{itemize}
	\item $m(\pi) > 0$,
  \item  $\iota_\infty(e)\cdot \pi^p_f \neq 0$ and
	\item  $\pi_p = \tau$.
\end{itemize}
We call two tame levels $\widetilde{K} \subset \widetilde{G}(\A^p_f)$ and $K \subset G(\A^p_f)$ globally Langlands compatible if the idempotents $e_{\widetilde{K}}$ and $e_K$ are.
\end{defn}

\begin{prop}
Let $\{g_i: i \in I \}$ be a set of representatives for $\widetilde{G}(F_v)/G(F_v)F_v^*$ and $\widetilde{K}\subset \widetilde{G}(F_v)$ be a compact open subgroup. Define the compact open subgroup 
$$K:=\bigcap_{i\in I} g_i^{-1} (\widetilde{K}\cap G(F_v)) g_i$$ 
of $G(F_v)$. Then $\widetilde{K}$ and $K$ are strongly locally Langlands compatible. 
\label{localcompact}
\end{prop}
\begin{proof} Let $\widetilde{\pi}$ be an irreducible admissible representation of $\widetilde{G}(F_v)$ which has a $\widetilde{K}$-fixed vector. Assume $\widetilde{\pi}|_{G(F_v)}$ decomposes as $\bigoplus_{i=1,..n} c \pi_i$. As $\widetilde{\pi}^{\widetilde{K}\cap G(F_v)}\neq 0$ there exists $\pi_i$ in the restriction, which has a non-zero $\widetilde{K}\cap G(F_v)$ fixed vector. We may assume without loss of generality that $\pi_i=\pi_1$. Now for any $\pi_k$ occurring in the restriction there exists $g_j \in \{g_i: i \in I\}$ such that $\pi_k \cong {^{g_j}\pi_1}$. Therefore $\pi_k$ has a fixed vector under $g_j^{-1} (\widetilde{K}\cap G(F_v)) g_j$.
\end{proof}

\begin{prop}
\begin{enumerate}
	\item Given an idempotent $\widetilde{e} \in C^\infty_c(\widetilde{G}(F_v),\Qbar)$, there exists an idempotent $e \in C^\infty_c(\widetilde{G}(F_v),\Qbar)$, which is strongly locally Langlands compatible with $\widetilde{e}$. 
	\item If $\widetilde{e}$ is a special idempotent associated with a supercuspidal Bernstein component (as in Section 3 of \cite{bushnell}), then there exists a special idempotent $e$ associated to a finite set of Bernstein components, which is strongly locally Langlands compatible with $\widetilde{e}$.
\end{enumerate}
\label{special}
\end{prop}
\begin{proof}
The first part follows from the previous proposition and the fact that there exists $\widetilde{K}\subset \widetilde{G}(F_v)$ such that $e_{\widetilde{K}} \cdot \widetilde{e} = \widetilde{e}$. 
For the second part note that a supercuspidal Bernstein component for $\widetilde{G}(F_v)$ in fact determines an $L$-packet $\Pi$. Indeed for any unramified character $\chi$ the $L$-packet of $\widetilde{\pi}\otimes \chi$ is the same as the $L$-packet of $\widetilde{\pi}$. We let $\Sigma_v$ be the set of the finitely many Bernstein components which occur in $\Pi$. Then $e:=e(\Sigma_v)$, the special idempotent attached to $\Sigma_v$, does the job. 
\end{proof}

\begin{prop} Assume the quaternion algebra $B$ ramifies at a finite place.
Let $\widetilde{K}=\prod_v {\widetilde{K}_v}$ be a compact open subgroup of $\widetilde{G}(\A^p_f)$. Then there exists a compact open subgroup  $K \subset G(\A^p_f)$, which is globally Langlands compatible with $\widetilde{K}$.
\label{gLC}
\end{prop}

\begin{proof}
Let $S$ be a finite set of places containing $p$ and $S_B$ and such that for $v\notin S$, $\widetilde{K}_v=\GL_2(\mathcal{O}_{F_v})$. Define the compact open subgroup $K:=\prod_v{K_v} \subset G(\A_f^p)$, where $K_v := \SL_2(\mathcal{O}_{F_v})$ for all $v \notin S$ and for all finite places $v\in S \backslash \{p\}, $ $K_v$ is the compact open defined in Proposition \ref{localcompact}. For an automorphic representation $\widetilde{\pi}$ of $\widetilde{G}(\A)$ with $(\widetilde{\pi}^p_f)^{\widetilde{K}} \neq 0$ and an element $\tau \in \Pi(\widetilde{\pi}_p)$ define the set 
\[ Y(\widetilde{\pi}, \tau) := \{ \pi \in \Pi(\widetilde{\pi}) \ | \ (\pi^p_f)^K\neq 0, \pi_p=\tau \}.\]
From Proposition 4.14. it follows that 
\[
 Y(\widetilde{\pi}, \tau)= \{\pi \in \Pi(\widetilde{\pi}) \ | \ \pi_v = \pi_v^0 \ \forall v \notin S , \pi_p = \tau \}.
\]
In order to show that there exists $\pi \in Y(\widetilde{\pi},\tau)$ with $m(\pi) >0$, let $\pi \in Y(\widetilde{\pi}, \tau)$ be arbitrary. If $m(\pi)>0 $ we are done, if not Proposition \ref{switch} implies that by changing the local representation $\pi_v$ at a finite place $v \in S_B$ we can pass to a representation~$\pi'$ which is automorphic and is obviously still in $Y(\widetilde{\pi},\tau)$. 
\end{proof}

\section{$p$-adic Langlands functoriality}
We use the notation of Section \ref{sec: zoo}. In particular, we assume that the base field~$F$ is equal to $\Q$. If $\D$ is an eigenvariety of idempotent type for $G$, i.e., $\D$ is the eigenvariety associated to a datum $(\mathcal{W}, M(e)_\bullet, \alpha_\bullet, \mathcal{H}_S, u)$, define $S(\D):= S$, the finite set of places used to build the Hecke algebra $\mathcal{H}_S$. Similarly define $S(\widetilde{\D})$ for an eigenvariety $\widetilde{\D}$ for $\widetilde{G}$.
\subsection{Definition}
The following definition of a $p$-adic Langlands transfer from $\widetilde{\D}$ to $\D$ is analogous to the one given in Section 3.4 of \cite{pjw} for the $p$-adic endoscopic transfer on unitary groups.
\begin{defn} Let $\widetilde{\D}$ (resp. $\D$) be an eigenvariety of idempotent type for $\widetilde{G}$ (resp. for $G$). Assume that $S(\D)= S(\widetilde{\D})=:S$. A  morphism $\zeta:\widetilde{\D} \rightarrow \D$ is called a $p$-adic Langlands functorial transfer if the diagrams
	
$$\xymatrix{
\widetilde{\D}  \ar[d]^{\widetilde{\omega}} \ar[r]^\zeta &\D \ar[d]^{\omega} \\
\widetilde{\mathcal{W}} \ar[r]^\mu &\mathcal{W} }
	\hspace{1cm}
	\xymatrix{
\mathcal{H}_S \ \ar[d] \ar@{^{(}->}[r]^\lambda & \widetilde{\mathcal{H}}_S \ar[d]\\
  \mathcal{O}(\D) \ar[r]^{\zeta^*}  &\mathcal{O}(\widetilde{\D}) }  $$ 
  commute.	
\label{padictransfer}
\end{defn}
\begin{lem} If $\zeta$ exists it is unique.
\label{uniquetransfer}
\end{lem}
\begin{proof}
As $\widetilde{\D}$ and $\D$ are reduced, $\zeta$ is determined by the induced map on $\Qbar_p$-points. However for $x \in \widetilde{\D}(\Qbar_p)$, $\psi_{\zeta(x)} = \psi_x \circ \lambda$ and $\omega(\zeta(x)) = \mu(\widetilde{\omega}(x))$. By Lemma \ref{7.2.7} the pair $(\psi_{\zeta(x)}, \omega(\zeta(x)))$ determines the point $\zeta(x) \in \D(\Qbar_p)$ uniquely. 
\end{proof}
\begin{rem} The target eigenvariety $\D$ is by no means unique. There can be $p$-adic transfers from $\widetilde{\D}$ to many different eigenvarieties.
\end{rem}

\subsection{Construction of the $p$-adic transfer -- the general case}

As before let $\widetilde{e}= \prod_l \widetilde{e}_l \in C_c^{\infty}(\widetilde{G}(\A_f^p),\Qbar)$ be an idempotent and $S$ a set of bad places. 
Let $\widetilde{K}= \prod_l \widetilde{K}_l \subset \widetilde{G}(\A_f^p)$ be a compact open such that $\widetilde{e}\cdot e_{\widetilde{K}}  =e_{\widetilde{K}} \cdot \widetilde{e}  = \widetilde{e}$ and let $K \subset G(\A_f^p)$ be a tame level which is globally Langlands compatible with $\widetilde{K}$, e.g.\ the one constructed in Proposition \ref{gLC}. Then $\widetilde{e}$ and $e:=e_K$ are also globally Langlands compatible.\footnote{ If for some $l \in S$, $\widetilde{e}_l$ happens to be the idempotent associated to a supercuspidal representation of $\widetilde{G}(\Q_l)$ we may alternatively choose $e_l$ as in Proposition \ref{special}.(2).} Let $\widetilde{\D}=  \D_{u_0}(\widetilde{e})$ be the eigenvariety of idempotent type $\widetilde{e}$ for $\widetilde{G}$ and $\mathcal{D}:=\D_{u_0}(e)$ the eigenvariety of idempotent type $e$ for $G$. We have all the ingredients to prove the existence of a $p$-adic Langlands functorial transfer from $\widetilde{\D}$ to $\D$. In Section \ref{subsec: aux} we constructed two auxiliary eigenvarieties $\D' $ and $\D''$. They fit into the following commutative diagram
$$\xymatrix{
\widetilde{\D} \ar[d]^{\widetilde{\omega}} \ar[r]^{\zeta'} & \D' \ar[d]^{\omega'} &  \D'' \ar[d]^{\omega''} \ar[r] &\D \ar[d]^{\omega} \\
\widetilde{\mathcal{W}} \ar[r]^{\operatorname{id}} & \widetilde{\mathcal{W}} \ar[r]^{\operatorname{id}}  &\widetilde{\mathcal{W}}  \ar[r]^\mu &\mathcal{W} }.$$ 
Recall from Proposition \ref{Z'e} that $\D'$ is an eigenvariety for $\mathcal{Z}'_{\widetilde{e}}$. As before we denote by $Z' $ the Zariski-dense and accumulation set of points in $\D'(\Qbar_p)$ defined by $ \mathcal{Z}'_{\widetilde{e}}$.
We construct a morphism $\xi$ from $\D'$ to $\D''$ by first establishing a map on classical points and then interpolating it.
\begin{prop}
There exists an injection 
$\xi:Z'\hookrightarrow \D''(\Qbar_p)$ such that  
$$(\psi_{\xi(z)}, \omega''(\xi(z))) = (\psi_{z}, \omega'(z)) \ \text{ for all } z \in Z'. $$
\end{prop}
\begin{proof}
Let $z \in Z'$ be a point. By definition there exists a $p$-refined automorphic representation $(\widetilde{\pi}, \widetilde{\chi})$ of $\widetilde{G}(\A)$ of weight $\widetilde{\omega}(z)$ such that $\widetilde{e}\cdot \widetilde{\pi}_f^p \neq 0$ and
$$\psi_{(\widetilde{\pi},\widetilde{\chi})}|_\mathcal{H}= \psi_z. $$
By Proposition \ref{gLC} there exists an automorphic representation $\pi$ of $G(\A)$, such that $(\pi_f^p)^K\neq 0$. After possibly changing the representation $\pi_p$ within the local $L$-packet we can guarantee that $\chi:= \widetilde{\chi}|_{\Sigma^+}$ occurs in $\pi_p^I$. Furthermore $\pi_{\infty}$ has weight $\mu(\omega(z))$. In particular  
$$ \psi_z = \psi_{(\pi,\chi)} $$
and the pair $(\psi_z, \mu(\omega(z)))$ corresponds to a point $y \in \D(\Qbar_p)$, necessarily unique by Lemma \ref{7.2.7}.
By construction $\omega(y)=\mu(\omega'(z))$. We define $\xi(z) \in \D''(\Qbar_p)$ to be the unique point corresponding to $(y,\omega'(z))$.

\end{proof}

\begin{rem} We set up the definition of global Langlands compatibility to have maximal flexibility at $p$ and so that it can easily be adapted to other groups. For our purpose we could have omitted the condition of choosing the member of the $L$-packet at $p$ arbitrarily in Definition \ref{GLCdef} for the following reason. The local $L$-packet attached to an unramified principal series representation $\tau$ of $\GL_2(\Q_p)$ has size at most two. It is of size two precisely when up to an unramified twist $\tau \cong Ind(\eta,1)$ for a quadratic character $\eta$. Assume this is the case and denote by $\tau_1$ and $\tau_2$ the two members of the $L$-packet attached to $\tau$. 
The representation $\tau$ has two distinct refinements $\chi_1 = (\eta,1)$ and $\chi_2=(1,\eta)$. But $\chi_1((p^{-1},p))= -1 = \chi_2((p^{-1},p))$, i.e.\ the refinements agree when restricted to $\Sigma^+$. Furthermore $\dim \tau_1^I = \dim \tau_2^I = 1$ (see Proposition 3.2.12 of \cite{lansky}), so each member of the $L$-packet sees the same eigenvalues. This implies that in the proof of the previous proposition there is in fact no need to change the representation at $p$. 
\end{rem}

\begin{prop}
There exists a closed immersion $\xi: \D' \hookrightarrow \D''$ such that the diagrams 
$$\xymatrix{
\D' \ar[rd]_{\omega'}\ar[rr]^{\xi} & & \ar[ld]^{\omega''}\D'' \\
&  \widetilde{\mathcal{W}}  &  }
\hspace{1cm}\xymatrix{
& \mathcal{H}_S \ar[ld] \ar[rd] &\\
\mathcal{O}(\D'') \ar[rr]_{\xi^*} & & \mathcal{O}(\D') }
 $$  
commute.
\end{prop}
\begin{proof}
Using the inclusion $Z' \hookrightarrow \D''(\Qbar_p)$ from the previous proposition we simply apply Proposition \ref{interpolation} to get the desired morphism $\xi: \D' \rightarrow \D''$.  
\end{proof}

Putting all the morphisms together we have proved
\begin{thm}
There exists a $p$-adic Langlands transfer $\zeta:\widetilde{\D} \rightarrow \D$. It has the additional property that it sends a classical point $z \in Z$ to a classical point in $\D$.
\label{p-adictransfer}
\end{thm}

\begin{corollary} Let $\widetilde{e}$ and $e$ be as above. Let $\D_u(\widetilde{e})$ be an eigenvariety for $\widetilde{G}$ where $u=\mathbf{1}_{\operatorname{[\widetilde{I}z\widetilde{I}]}} \otimes 1_{\widetilde{H}_{ur}} \in \widetilde{\mathcal{H}}_S$ with $z \in \widetilde{\Sigma}^{++}$. There exists a $p$-adic Langlands functorial transfer
$$\D_u(\widetilde{e}) \rightarrow  \D_{u_0}(e).$$
\end{corollary}
\begin{proof}
Lemma \ref{changeu} gives an isomorphism from $\D_u(\widetilde{e}) $ to $\D_{u_0}(\widetilde{e})$. Composing the $p$-adic transfer $\zeta$ constructed above with this isomorphism then gives a $p$-adic transfer from $\D_u(\widetilde{e})$ to $\D_{u_0}(e)$. 
\end{proof}

\subsection{A special $p$-adic transfer}
Assume that $\widetilde{K}_{st} :=\prod_{l\neq p} K_l \subset \widetilde{G}(\A_f^{p})$ is a compact open subgroup such that for some finite place $q \in S_B$ we have $K_{q}\cong \mathcal{O}^*_{B_{q}}$ where~$\mathcal{O}_{B_{q}}$ is a maximal order in $B_{q}$. Define $K_{st}:= \widetilde{K}_{st}\cap G(\A_f^{p})$. Define $\widetilde{e}_{st}:=e_{\widetilde{K}_{st}} $ and $e_{st}:= e_{K_{st}}$.
\begin{thm}
There exists a $p$-adic Langlands transfer $\zeta_{st}:\D({\widetilde{e}_{st}}) \rightarrow \D({e_{st}})$.
\end{thm}
\begin{proof} Any $p$-refined automorphic representation $(\widetilde{\pi}, \widetilde{\chi})$ of $\widetilde{G}(\A)$ such that $\widetilde{e}_{st} \cdot \widetilde{\pi}_f^{p} \neq 0 $ gives rise to a global $L$-packet of $G(\A)$ which is stable. This can be seen as follows. If $\widetilde{\tau}:= JL(\widetilde{\pi})$ denotes the Jacquet--Langlands transfer of $\widetilde{\pi}$ then $\widetilde{\tau}_{q}$ is isomorphic to an unramified twist of the Steinberg representation. However if $\widetilde{\tau}$ comes from a Gr{\"o}{\ss}encharacter, then $\widetilde{\tau}_q$ is either a principal series representation or supercuspidal. Therefore $\widetilde{\pi}$ cannot give rise to an endoscopic packet. 

It is now easy to verify that $\widetilde{K}_{st}$ and $K_{st}$ are globally Langlands compatible. Namely let $\widetilde{\pi}$ be as above and $\Pi(\widetilde{\pi})$ the associated $L$-packet. Then $\Pi(\widetilde{\pi})$ contains an element $\pi$ such that $(\pi_f^{p})^{K_{st}}\neq 0 $. As the multiplicity within a stable $L$-packet is constant and positive $\pi$ is automorphic. As we can select $\pi_p$ arbitrarily we see that~$\widetilde{K}_{st}$ and $K_{st}$ are indeed globally Langlands compatible. The rest of the proof is as in the general case.  
\end{proof}
The above gives an example of a $p$-adic transfer between eigenvarieties where the local idempotents $\widetilde{e}_l$ and $e_l$ at the bad places are locally Langlands compatible but not necessarily strongly locally Langlands compatible.

\bibliography{p-adic_LL}
\bibliographystyle{plain}

\enddocument